\documentclass[a4paper,11pt]{article}

\usepackage{bm}
\usepackage[english]{babel}
\usepackage{hyperref}
\usepackage{amssymb,amsmath,amsthm}
\usepackage{enumerate}
\usepackage{enumitem}
\usepackage{graphicx, color}
\usepackage{tikz}
\usetikzlibrary{calc,decorations.pathmorphing,decorations.text}
\usepackage{subcaption}
\usepackage[hmargin=2.5cm,vmargin=3cm]{geometry}
\usepackage{mathtools, braket}
\usepackage[normalem]{ulem}
\usepackage{indentfirst}
\allowdisplaybreaks[4]
\usepackage[nameinlink]{cleveref}

\usepackage{hyperref}
\hypersetup{
	colorlinks=true,
	linkcolor=black,
	citecolor=black,
    urlcolor=black
 }

\theoremstyle{plain}
\newtheorem{thm}{Thm}[section]

\newtheorem{theorem}[thm]{Theorem}
\newtheorem{lemma}[thm]{Lemma}
\newtheorem{corollary}[thm]{Corollary}
\newtheorem{proposition}[thm]{Proposition}
\newtheorem{conjecture}[thm]{Conjecture}

\newtheorem{observation}[thm]{Observation}
\newtheorem{definition}[thm]{Definition}

\renewcommand\ge\geqslant
\renewcommand\geq\geqslant
\renewcommand\le\leqslant
\renewcommand\leq\leqslant
\newcommand\vph{\varphi}

\usepackage{caption}
\usetikzlibrary{matrix}
\usetikzlibrary{decorations.markings}
\tikzstyle{vertex}=[circle, draw, fill=black!50,
inner sep=0pt, minimum width=4pt]
\tikzset{->-/.style={decoration={
			markings,
			mark=at position .5 with {\arrow{>}}},postaction={decorate}}}

\tikzstyle{bigblue}=[color=blue, very thick, >=stealth]
\tikzstyle{lightblue}=[color=blue, thin, >=stealth]
\tikzstyle{bigred}=[color=red, very thick, >=stealth]
\tikzstyle{lightred}=[color=red, thin, >=stealth]
\tikzstyle{biggreen}=[color=black!30!green, very thick, >=stealth]
\tikzstyle{lightgreen}=[color=black!30!green,  thin, >=stealth]

\newlength{\bibitemsep}
\newlength{\bibparskip}
\let\oldthebibliography\thebibliography
\renewcommand\thebibliography[1]{
  \oldthebibliography{#1}
  \setlength{\parskip}{\bibitemsep}
  \setlength{\itemsep}{\bibparskip}
}

\begin{document}
	\title{Planar Graphs with Homomorphisms to the 9-cycle}

    \author{Daniel W. Cranston\thanks{Virginia Commonwealth University, Dept.~of Computer Science; Richmond, VA, USA; \texttt{dcranston@vcu.edu}}, 
    Jiaao Li\thanks{School of Mathematical Sciences and LPMC, Nankai University, Tianjin 300071, China; \texttt{lijiaao@nankai.edu.cn}}, 
    Zhouningxin Wang\thanks{School of Mathematical Sciences and LPMC, Nankai University, Tianjin 300071, China; \texttt{wangzhou@nankai.edu.cn}},
    and Chunyan Wei\thanks{Center for Combinatorics and LPMC, Nankai University, Tianjin 300071, China; \texttt{yan1307015@163.com}}}

    \date{ }
    
	\maketitle

\begin{abstract}
We study the problem of finding homomorphisms into odd cycles from planar graphs with high odd-girth. The Jaeger-Zhang conjecture states that every planar graph of odd-girth at least $4k+1$ admits a homomorphism to the odd cycle $C_{2k+1}$. The $k=1$ case is the well-known Gr\"otzsch's $3$-coloring theorem. 
For general $k$, in 2013 Lovász, Thomassen, Wu, and Zhang showed that it suffices to have odd-girth at least $6k+1$.  Improvements are known for $C_5$ and $C_7$ in [Combinatorica 2017, SIDMA 2020, Combinatorica 2022].  
For $C_9$ we improve this hypothesis by showing that it suffices to have odd-girth 23. 
Our main tool is a variation on the potential method applied to modular orientations. 
This allows more flexibility when seeking reducible configurations. The same techniques also prove some results on circular coloring 
of signed planar graphs.
\end{abstract} 

\textbf{Keywords.} homomorphism, modular orientation, potential method.

\section{Introduction} 

\subsection{Homomorphisms to odd cycles}
A \emph{homomorphism} (or \emph{map}) of a graph $G$ to another graph $H$ is a function $\varphi: V(G)\to V(H)$ that preserves 
adjacency.  Note that a graph $G$ is $k$-colorable if and only if $G$ has a homomorphism to $K_k$.  So the study of graph homomorphisms extends and strengthens many results on graph coloring.  Gr\"otzsch's $3$-coloring theorem famously states that every planar graph of odd-girth at least $5$ admits a homomorphism to $C_3$. Recall that the \emph{odd-girth} of a graph is the length of its shortest odd cycle.  The following conjecture, which generalizes Gr\"otzsch's theorem, concerns homomorphisms of a graph $G$ to odd cycles when $G$ is planar with large odd-girth.
(It is worth noting that $C_{2k+3}$ maps to $C_{2k+1}$ for all $k\ge 1$, but not vice versa.  So proving that a graph $G$ maps to
$C_{2k+3}$ is stronger than proving that $G$ maps to $C_{2k+1}$.)

\begin{conjecture}
    \label{conj:homomorphism to odd cycles}\label{jaeger-planar-conj}
	Every planar graph of odd-girth at least $4k+1$ admits a homomorphism to $C_{2k+1}$.
\end{conjecture}

If true, \Cref{jaeger-planar-conj} is best possible. To see this, begin with a cycle $C_{4k-1}$; add a new vertex $v$ and vertex-disjoint paths from $v$ to all vertices of the cycle, with each path of length $2k-1$.  The resulting graph is planar with odd-girth $4k-1$, but it is straightforward to check\footnote{Observe that in a hypothetical homomorphism to $C_{2k+1}$ no vertex on the original cycle $C_{4k-1}$ has the same image as $v$; thus, this $C_{4k-1}$ must map into a path, which is impossible.} that this graph has no homomorphism to $C_{2k+1}$.

For planar graphs, \Cref{jaeger-planar-conj} is the dual version of Jaeger's 1982 Circular Flow Conjecture~\cite{J88}, which asserts that every $4k$-edge-connected multigraph admits a circular $\frac{2k+1}{k}$-flow.  A \emph{circular $\frac{p}{q}$-flow}, given positive integers $p$ and $q$ with $p\geq 2q$,
of a graph $G$ is a flow $(D, f)$ such that $q\leq |f(e)|\leq p-q$ for every edge $e\in E(G)$. So a circular $\frac{2k+1}k$-flow requires that $|f(e)|\in\{k,k+1\}$ for all $e$. In 2018, Jaeger's conjecture was disproved \cite{HLWZ18} for all $k\ge 3$. However, the known counterexamples are all non-planar; so it is still possible that \Cref{jaeger-planar-conj} is true. This problem remains open. 

Let $G$ be a planar graph with girth $g$. In 1996, Ne\v{s}et\v{r}il and Zhu~\cite{NZ96} proved that $G$ maps to $C_{2k+1}$ whenever $g\ge 10k-4$.
In 2001, Zhu~\cite{Z01} improved this hypothesis to $g \ge 8k-3$. 
And in 2004, Borodin, Kim, Kostochka, and West~\cite{BKKW04} further improved this hypothesis to $g \ge \frac{20k-2}{3}$. Finally, in 2013, Lov\'asz, Thomassen, Wu, and Zhang~\cite{LTWZ13} proved that $G$ maps to $C_{2k+1}$ if $G$ has odd-girth at least $6k+1$.

The general result~\cite{LTWZ13} for odd-girth at least $6k+1$ has been improved for small $k$.  Let $g_o$ denote the odd-girth of $G$.
When $k = 2$, \cite{LTWZ13} implies that $G$ maps to $C_5$ when $g_o\ge 13$.
Dvo\v{r}\'ak and Postle~\cite{DP17} and Cranston and Li~\cite{CL20} improved this to $g_o\ge 11$.
When $k = 3$, \cite{LTWZ13} implies that $G$ maps to $C_7$ when $g_o\ge 19$.
Cranston and Li~\cite{CL20} and Postle and Smith-Roberge~\cite{PS22} both improved this to $g_o\ge 17$.
Continuing with this line of study, when $k=4$ we improve the hypothesis of~\cite{LTWZ13} from $g_o\ge 25$ to $g_o\ge 23$. That is, we prove the following.

\begin{theorem}\label{thm:23_homomorphism_to_C9}
    Every planar graph of odd-girth at least $23$ admits a homomorphism to $C_{9}$. 
\end{theorem}

\noindent
Recall that \Cref{jaeger-planar-conj} posits that the hypothesis $g_o\ge 23$ in this result can be improved to $g_o\ge 17$, which if true would be best possible.

\subsection{Modular orientations}
For a plane graph $G$ and its planar dual $G^*$, the dual of each cycle $C$ in $G$ is a minimal edge cut $C^*$ in $G^*$. So $G^*$ has odd-girth at least 23 if and only if in $G$ every minimal edge cut of odd size has size at least 23.  Such a graph $G$ has \emph{odd-edge-connectivity} at least 23.
We find it more convenient to prove \Cref{thm:23_homomorphism_to_C9} in its dual formulation. But before stating this version we need a new definition, proposition, and lemma.  The proposition is an easy exercise\footnote{Given a map $\vph$ from $G$ to $C_{2k+1}$, the faces $f$ of $G^*$ take colors $\vph^*(f)$ from the corresponding vertices of $G$.  Now each edge $e$ of $G^*$, with face $f_1$ on the left and face $f_2$ on the right, has flow value $k(\vph^*(f_1)-\vph^*(f_2))$. It is straightforward to check that this process yields a circular $\frac{2k+1}k$-flow of $G^*$.}, and the lemma was proved by Jaeger~\cite{J84} in 1984.

\begin{definition}\label{def:modular_orientation}
    Given an orientation $D$ of $G$, if the outdegree of $v$ is congruent to the indegree of $v$ modulo $2k+1$ (i.e., $d^+_D(v)-d^-_D(v)\equiv 0\pmod {2k+1}$) for every vertex $v\in V(G)$, then we call $D$ a \emph{modular $(2k+1)$-orientation}.   
\end{definition}

\begin{proposition}\label{prop:circular_flow_homomorphims_C2k+1}
    A plane graph $G$ admits a homomorphism to $C_{2k+1}$
    if and only if its planar dual $G^*$ admits a circular $\frac{2k+1}{k}$-flow.
\end{proposition}

\begin{lemma}\label{lem:modular_orientation_circular_flow}{\em \cite{J84}}
	A graph $G$ has a circular $\frac{2k+1}{k}$-flow if and only if $G$ admits a modular $(2k+1)$-orientation.
\end{lemma}

Suppose that $G^*$ is a plane graph with odd-girth at least $23$. Now $G$ has odd-edge-connectivity at least $23$. By~\Cref{prop:circular_flow_homomorphims_C2k+1}, $G^*$ admits a homomorphism to $C_9$ if and only if $G$ admits a 
circular $\frac{9}{4}$-flow. And by~\Cref{lem:modular_orientation_circular_flow}, this is equivalent to $G$ having a modular $9$-orientation.  Thus, we reach the following dual formulation of~\Cref{thm:23_homomorphism_to_C9}, which is what we will prove.

\medskip
\noindent
\textbf{{\Cref*{thm:23_homomorphism_to_C9}$'$}.}\label{thm:23_homomorphism_to_C9-restated} 
\textit{Every planar graph with odd-edge-connectivity at least $23$ admits a modular 9-orientation.}
\medskip

In fact, we will consider a more general type of orientation, which includes modular $k$-orientations as a special case.
But we defer this definition until \Cref{subsec:strongly_Z9_connected}.  In the remainder of this section, we 
present much of the intuition behind our proofs.

When proving graph coloring results, it is typical to phrase hypotheses in terms 
of \emph{maximum average degree}, which is the maximum of the average degrees of all subgraphs. 
Given a bound on maximum average degree, we can find some reducible configurations with good properties. 
For instance, given a graph $G$, we form a smaller graph $G'$ by deleting a suitable subgraph $H$, 
we color $G'$ by induction, and we can then extend every proper coloring of $G'$ to $G$. 
To prove Theorem~\ref{thm:23_homomorphism_to_C9-restated}$'$, we will construct a modular 9-orientation. Thus we seek a concept, in the context of modular orientations, that parallels reducible configurations for graph coloring.

The dual of edge deletion in a planar graph $G$ is edge \emph{contraction} in its planar dual $G^*$. For an edge $e\in E(G)$, \emph{contracting} $e$ means identifying its two endpoints and deleting the resulting loop. We denote the resulting graph by $G/e$. Now the graph $G/H$, where $H\subseteq G$, is formed from $G$ by contracting all edges of $E(H)$.
Thus, our reducible configurations in this context will be subgraphs $H$ such that we can form $G'$ by contracting $H$,
find a modular orientation for $G'$, and then extend this orientation to a modular orientation for $G$. Intuitively, such a process is easier when $G$ has higher edge-connectivity.
But rather than working directly with edge-connectivity, we rely on a weight function introduced in~\cite{CL20}.  
Analogous to maximum average degree, the weight of a graph is defined as a minimum over all vertex-partitions of 
the number of edges between parts, minus a term linear in the number of parts. Intuitively, a graph with higher 
edge-connectivity has higher weight.  This weight function is motivated by the Nash-Williams--Tutte Theorem~\cite{N61,T61}, which characterizes graphs with $k$ edge-disjoint spanning trees, for each positive integer $k$.

\begin{definition}\label{def:weight_function}
For a graph $G$ and a partition $\mathcal{P}$ of $V(G)$ with parts $P_1,P_2,\ldots,P_t$, the \emph{weight function} is defined 
as follows:
\begin{equation}\label{equ:weight_function}w_G(\mathcal{P}):=\sum_{i=1}^{t}d(P_i)-23t+42.
\end{equation}	
Here $d(P_i)$ is the number of the edges of $E(G)$ with exactly one endpoint in $P_i$.

The \emph{weight} of a graph $G$ is the minimum weight over all its partitions, i.e., $w(G):=\min\limits_{\mathcal{P}} w_{G}(\mathcal{P})$.
\end{definition}

We will prove our main results for all planar graphs $G$ with $w(G)\ge 0$, excluding a few well-understood exceptional graphs. To help build intuition, first consider the trivial partition $\mathcal{P}_0$, where each vertex forms its own part.
Here the inequality $0\leq w(G)\le w_G(\mathcal{P}_0)=2e(G)-23v(G)+42$
implies that $e(G)\geq \frac{1}{2}(23v(G)-42)$.
It is helpful to observe that $w(G/H)\ge w(G)$ for all graphs $G$ and connected subgraphs $H$. Thus, we aim to find a ``good'' subgraph $H$, recursively find the desired orientation for $G/H$, and then extend this orientation to $G$, due to our choice of $H$. However, this
is not always possible. Sometimes a contractible graph $H$ is not a subgraph of $G$, but $G$ contains some $H'$ formed from $H$ by deleting a few edges. In this case, we can often ``lift'' pairs of edges outside $E(H')$ to simulate the missing edges of $H$.

It turns out that this process of finding good reductions, via contraction (and sometimes lifting)---which is the key to proving Theorem~\ref{thm:23_homomorphism_to_C9-restated}$'$---also yields a result on $r$-flows in \emph{signed} planar graphs (see 
\Cref{thm:23_circular_flow_signed_planar_graphs}). Thus, we extract from the proof of Theorem~\ref{thm:23_homomorphism_to_C9-restated}$'$ a key structural result, \Cref{thm:main_statement_SZ9_k8} (see below), the proof of which requires most of our work.

It is good now to explain the choice of the constants $23$ and $42$ in the definition of the weight function.
The number $23$ in the weight function comes from our desired edge-connectivity. Given a $23$-edge-connected graph $G$, clearly $w(G)\ge 42$, since for every partition $\{P_1,P_2,\ldots,P_t\}$ of $V(G)$, each term $d(P_i)$ in the sum is at least $23$; thus, our results will apply.  We will prove our main theorem by induction, and it is convenient for the inductive hypothesis to be satisfied by certain graphs of small order that are not 23-edge-connected. This motivates
introducing the additive constant 42, which ensures\footnote{We have some flexibility in picking this constant $c$. 
We need $c\ge 42$ to ensure that $w(G)\ge 0$ for every 4-vertex graph with at least 25 edges; see~\Cref{lem:SZ9_small_graphs}\ref{lem:4small_SZ_9}. 
And we need $c\le 45$ to get a contradiction via discharging in Section~\ref{sec:discharging}; see~\Cref{equ:discharge}.
Proving~\Cref{thm:main_statement_SZ9_k8} with a larger value of $c$ would be a stronger result.  
However, proving such a statement would likely require much more work, since the analogous set $\mathcal{N}\cup\mathcal{W}^*$ of exceptional graphs would
be larger. We prefer to keep the proof as simple as possible, which is why we chose $c=42$.} that we will have $w(G)\ge 0$ for these various graphs of small order.  

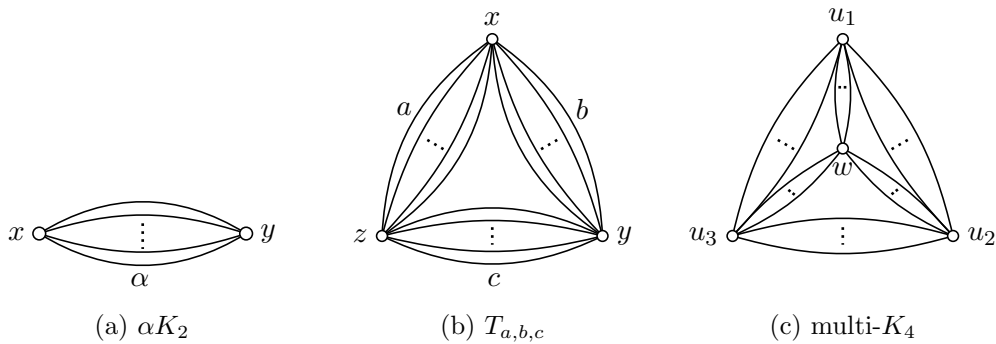
\begin{figure}[tbp]
	\centering
	\begin{subfigure}[t]{.28\textwidth}
        \centering
		\begin{tikzpicture}[scale=.68]		
			\draw [line width=1pt, dotted] (0, 0.2) to (0, -0.3);
			\draw(0.4, -0.9) node[left=0.5mm]  {$\alpha$};
			\draw [bend left=18, line width=0.6pt, black] (-2,0) to (2,0);%xy
			\draw [bend right=18, line width=0.6pt, black] (-2,0) to (2,0);
			\draw [bend left=32, line width=0.6pt, black] (-2,0) to (2,0);
			\draw [bend right=32, line width=0.6pt, black] (-2,0) to (2,0);
			
			\draw [fill=white,line width=0.6pt] (-2,0) node[left=0.5mm] {$x$} circle (3.5pt);   
			\draw [fill=white,line width=0.6pt] (2,0) node[right=0.5mm] {$y$} circle (3.5pt);          
		\end{tikzpicture}
		\caption{$\alpha K_2$}
		\label{fig:aK2}     
	\end{subfigure}
	\begin{subfigure}[t]{.28\textwidth}
		\centering
		\begin{tikzpicture}[scale=.58]			
			\draw [rotate=240] [line width=1pt, dotted] (1, -0.7) to (1, -1.2);%xz			
			\draw [bend left=14, line width=0.6pt, black] (0,2) to  (-2.5,-2.5) ;%xz
			\draw [bend right=14, line width=0.6pt, black] (0,2) to  (-2.5,-2.5) ;
			\draw [bend left=26, line width=0.6pt, black] (0,2) to  (-2.5,-2.5) ;
			\draw [bend right=26, line width=0.6pt, black] (0,2) to  (-2.5,-2.5) ;

			\draw [line width=1pt, dotted] (0, -2.3) to (0, -2.7);%yz		
			\draw [bend left=14, line width=0.6pt, black]  (-2.5,-2.5)  to (2.5,-2.5);%yz
			\draw [bend right=14, line width=0.6pt, black]  (-2.5,-2.5)  to (2.5,-2.5);
			\draw [bend left=26, line width=0.6pt, black]  (-2.5,-2.5)  to (2.5,-2.5);
			\draw [bend right=26, line width=0.6pt, black]  (-2.5,-2.5)  to (2.5,-2.5);

			\draw [rotate=120] [line width=1pt, dotted](-1, -0.7) to (-1, -1.2);%xy
			\draw [bend left=14, line width=0.6pt, black] (0,2) to (2.5,-2.5); %xy
			\draw [bend right=14, line width=0.6pt, black] (0,2) to (2.5,-2.5);
			\draw [bend left=26, line width=0.6pt, black] (0,2) to (2.5,-2.5);
			\draw [bend right=26, line width=0.6pt, black] (0,2) to (2.5,-2.5);
			
			\draw [fill=white,line width=0.6pt] (0,2) node[above=0.5mm] {$x$} circle (3.5pt); 
			\draw [fill=white,line width=0.6pt] (2.5,-2.5) node[right=0.5mm] {$y$} circle (3.5pt);          
			\draw [fill=white,line width=0.6pt] (-2.5,-2.5) node[left=0.5mm] {$z$} circle (3.5pt);            
			\draw  (-1.5,0.4) node[left=0.5mm] {$a$};   
			\draw(0.5,-3.5) node[left=0.5mm] {$c$};  
			\draw  (2.5,0.4) node[left=0.5mm] {$b$};  
		\end{tikzpicture}
		\caption{$T_{a,b,c}$}
		\label{fig:Tabc}
	\end{subfigure}
    \begin{subfigure}[t]{.28\textwidth}
		\centering
		\begin{tikzpicture}[scale=.58]	
            %xz
			\draw [rotate=240] [line width=1pt, dotted] (1, -0.7) to (1, -1.2);			
			\draw [bend left=14, line width=0.6pt, black] (0,2) to  (-2.5,-2.5);
			\draw [bend right=16, line width=0.6pt, black] (0,2) to  (-2.5,-2.5);
			
             %yz
			\draw [line width=1pt, dotted] (0, -2.3) to (0, -2.7);		
			\draw [bend left=16, line width=0.6pt, black]  (-2.5,-2.5)  to (2.5,-2.5);
			\draw [bend right=16, line width=0.6pt, black]  (-2.5,-2.5)  to (2.5,-2.5);
			
            %xy
			\draw [rotate=120] [line width=1pt, dotted](-1, -0.7) to (-1, -1.2);
			\draw [bend left=16, line width=0.6pt, black] (0,2) to (2.5,-2.5); 
			\draw [bend right=14, line width=0.6pt, black] (0,2) to (2.5,-2.5);

            \draw[rotate=0] [line width=1pt, densely dotted] (-0.1,0.9) to (0.1,0.9);%wx
			\draw [bend left=14,line width=0.6pt] (0,-0.5) to (0,2);
			\draw [bend right=14,line width=0.6pt] (0,-0.5) to (0,2); 

            \draw[rotate=218] [line width=1pt, densely dotted] (-0.1,1.9) to (0.1,1.9);%wy
			\draw [bend left=10,line width=0.6pt] (0,-0.5) to (2.5,-2.5);
			\draw [bend right=14,line width=0.6pt] (0,-0.5) to (2.5,-2.5);

            \draw[rotate=142] [line width=1pt, densely dotted] (-0.1,1.9) to (0.1,1.9);%wz
			\draw [bend left=14,line width=0.6pt] (0,-0.5) to (-2.5,-2.5);
			\draw [bend right=10,line width=0.6pt] (0,-0.5) to (-2.5,-2.5);

            \draw [fill=white, line width=0.6pt] (0,2) node[above=0.5mm] {$u_1$} circle (3.5pt); %x
			\draw [fill=white, line width=0.6pt] (2.5,-2.5) node[right=0.5mm] {$u_2$} circle (3.5pt); %y         
			\draw [fill=white, line width=0.6pt] (-2.5,-2.5) node[left=0.5mm] {$u_3$} circle (3.5pt);  %z   
            \draw [fill=white, line width=0.6pt] (0,-0.5) node[below=0.5mm] {$w$} circle (3.5pt);       
			\draw [color=white](-1.5,0.4) node[left=0.5mm] {$a$};   
			\draw [color=white](0.5,-3.5) node[left=0.5mm] {$c$};  
			\draw [color=white](2.5,0.4) node[left=0.5mm] {$b$}; 
        \end{tikzpicture}
		\caption{multi-$K_4$}
		\label{fig:K4}
	\end{subfigure}
	\caption{\small The graphs $\alpha K_2$, $T_{a,b,c}$, and multi-$K_4$.}
	\label{fig:ak_2 and Tabc}
\end{figure}

But this additive term also brings
complications, since now our theorem's hypothesis will be satisfied by some small graphs that do not satisfy the desired conclusion. Thus, we must classify exactly which small graphs do and do not satisfy this conclusion.
So we define the following two families of graphs: $\mathcal{N}$ and $\mathcal{W}^*$. As shown in~\Cref{fig:ak_2 and Tabc}, 
the graph $\alpha K_2$ has vertex set $\{x,y\}$ with its vertices joined by $\alpha$ parallel edges ($\alpha\geq 1$), and $T_{a,b,c}$ has vertex 
set $\{x,y,z\}$ with its pairs of vertices joined by $a$, $b$, and $c$ parallel edges ($\min\{a,b,c\}\geq 1$).
\begin{eqnarray*}
	&\mathcal{N}:=\{\alpha K_2\colon \alpha\leq 7\}\cup\{T_{a,b,c}\colon a+b+c\leq 15\}\\
	&\mathcal{W}^*:=\{8K_2\}\cup\{T_{a,b,c}\colon a+b+c=16,\delta\geq 9\}
\end{eqnarray*}
As we will explain in more detail when we state~\Cref{thm:partition_SZ9_SC16} (see \Cref{sec:proofs of two theorems}),
all graphs in $\mathcal{N}$ are too sparse to satisfy our desired conclusion, so we will
need to explicitly exclude them from our main theorem.  Graphs in $\mathcal{W}^*$ will {satisfy} our conclusion, but only barely.
So we will need to handle them more delicately than other graphs to which our main result applies.

When specifying a graph formed by contraction, we will frequently use the following definition. For a graph $G$ and a partition $\mathcal{P}$ with parts $P_1,P_2,\ldots,P_t$ of $V(G)$, let $G/\mathcal{P}$ denote the graph formed from $G$ by identifying all the vertices of $G[P_i]$ for each $i\in [t]$. Let $|\mathcal{P}|$ denote the number of parts in $\mathcal{P}$ and observe that $v(G/\mathcal{P})=|\mathcal{P}|$.
We will often want to know the weight of $\alpha K_2$, $T_{a,b,c}$, and graphs with four vertices, so we record the following observation for easy reference. 

\begin{observation}\label{ob:weight_small_graphs}
    \begin{enumerate}[label=(\arabic*)]
    	\setlength{\itemsep}{0.1em}
    	\item Each graph $\alpha K_2$ satisfies $w(\alpha K_2)=2\alpha-4$.
    	\item Each graph $T_{a,b,c}$ without $8K_2$ satisfies $w(T_{a,b,c})=2(a+b+c)-27$. 
    	\item Each graph $G$ on four vertices without $8K_2$ satisfies $w(G)=2e(G)-50$. 
    \end{enumerate}
\end{observation}
\begin{proof}
For $\alpha K_2$, $w(\alpha K_2)=2\alpha-23\times 2+42=2\alpha-4$.

In the remaining two cases, every two vertices are joined by at most $7$ parallel edges, that is, we assume $\mu\le 7$, where $\mu$ is the maximum multiplicity over all edges in the graph we consider. For a graph $T_{a,b,c}$, let $\mathcal{P}_0$ be its trivial partition, and let $\mathcal{P}$ be an arbitrary partition with two parts. Since $\mu\le 7$, note that $w_{\scriptscriptstyle T_{a,b,c}}(\mathcal{P})\ge w_{\scriptscriptstyle T_{a,b,c}}(\mathcal{P}_0)+23-2\times 7>
w_{\scriptscriptstyle T_{a,b,c}}(\mathcal{P}_0)$.
Thus the minimum in the definition of $w(T_{a,b,c})$ is achieved by the trivial partition (with $3$ parts).   
For a graph $G$ with $v(G)=4$, let $\mathcal{P}_0$ be the trivial partition, let $\mathcal{P}_1$ be an arbitrary partition with three parts, and $\mathcal{P}_2$ be an arbitrary partition with two parts. Similar to the above discussion, we have $w_G(\mathcal{P}_0)<w_G(\mathcal{P}_1)$.
In fact, we also have $w_G(\mathcal{P}_2)\ge w_G(\mathcal{P}_0)+2\times 23-2\times 3\times7>w_G(\mathcal{P}_0)$. Hence, $w(G)=w_{G}(\mathcal{P}_0)=2e(G)-23v(G)+42$. 
\end{proof}

In particular, $w(\alpha K_2)\leq 10$ when $\alpha\leq 7$ and $w(8K_2)= 12$. Similarly, $w(T_{a,b,c})\leq 3$ when $a+b+c\leq 15$ and $w(T_{a,b,c})=5$ when $a+b+c=16$. 
Hence, all $G\in\mathcal{N}$ satisfy $w(G)\le\max\{3,10\}=10$ and all $G\in\mathcal{W}^*$ satisfy $w(G)\le\max\{5,12\}=12$. 
Now a short case analysis yields the following observation.

\begin{observation}\label{ob:partition_bound_notin}
    Given a graph $G$ and a partition $\mathcal{P}$ of $V(G)$, 
    \begin{enumerate}[label=(\arabic*)]
        \item\label{ob:3_12_N} if $w_{G}(\mathcal{P})\geq 11$, then $G/\mathcal{P}\notin\mathcal{N}$;
        \item\label{ob:1_14_WN} if $w_{G}(\mathcal{P})\geq 13$, then $G/\mathcal{P}\notin \mathcal{N}\cup\mathcal{W}^*$;
        \item\label{ob:4_5_N} if $w_{G}(\mathcal{P})\geq 4$ and $|\mathcal{P}|\geq 3$, then $G/\mathcal{P}\notin \mathcal{N}$;
        \item\label{ob:2_7_WN} if $w_{G}(\mathcal{P})\geq 6$ and $|\mathcal{P}|\geq 3$, then $G/\mathcal{P}\notin \mathcal{N}\cup\mathcal{W}^*$; 
        \item\label{ob:5_10-edge-connected} if $w_{G}(\mathcal{P})\geq 6$ and $G$ is $9$-edge-connected, then $G/\mathcal{P}\notin\mathcal{N}\cup\mathcal{W}^*$.
    \end{enumerate}
\end{observation}

When proving our main structural theorem by induction, it is helpful to know that no sequence of edge contractions yields a graph in $\mathcal{N}$; that is, for all partitions $\mathcal{P}$, we have $G/\mathcal{P}\notin\mathcal{N}$.  This motivates the following two closely related definitions.

\begin{definition}\label{def:graph_family}
	Let $G$ be a connected graph such that $v(G)\geq2$.
	\begin{enumerate}[label=(\roman*)]
		\item If $w(G)\geq0$, $G\notin \mathcal{N}$, and $G/\mathcal{Q}\notin \mathcal{N}\cup \mathcal{W}^*$ for every nontrivial partition $\mathcal{Q}$, then we call $G$ an \emph{$\mathcal{N}$-good} graph. 
		
		\item If $w(G)\geq0$ and $G/\mathcal{P}\notin \mathcal{N}\cup \mathcal{W}^*$ for every partition $\mathcal{P}$, then we call $G$ an \emph{$\mathcal{S}$-good} graph.
	\end{enumerate}
\end{definition}

From the definitions of $\mathcal{N}$-good and $\mathcal{S}$-good graphs, we note the following.
\begin{observation}\label{ob:W*_N_good}
    Every $\mathcal{S}$-good graph is also $\mathcal{N}$-good. Every graph in $\mathcal{W}^*$ is $\mathcal{N}$-good (but not $\mathcal{S}$-good).
\end{observation}

The second statement requires a bit of case analysis. We illustrate this with an example: the graph $T_{2,7,7}\in \mathcal{W}^*$. Obviously $T_{2,7,7}\notin \mathcal{N}$.
Note that $v(T_{2,7,7})=3$ and $w(T_{2,7,7})=5 >0$ by~\Cref{ob:weight_small_graphs}.
Additionally, every nontrivial partition $\mathcal{P}$ satisfies $T_{2,7,7}/\mathcal{P}\notin\{\alpha K_2: \alpha\leq 8\}$ as $\delta(T_{2,7,7})=9$. Furthermore, we also know that $|\mathcal{P}|=2$ for every nontrivial partition $\mathcal{P}$; thus $T_{2,7,7}/\mathcal{P}\notin\{T_{a,b,c}:a+b+c\leq 16\}$.
Hence, $T_{2,7,7}$ is $\mathcal{N}$-good. Similarly, noting that the condition of $\delta\geq 9$ guarantees that $\mu\leq 7$ for $T_{a,b,c}\in \mathcal{W}^*$, we can verify that the other graphs in $\mathcal{W}^*$  are also $\mathcal{N}$-good. 

We also remark that every $\mathcal{N}$-good (or $\mathcal{S}$-good) graph satisfies the hypotheses of~\Cref{thm:main_statement_SZ9_k8}. Thus, this result will facilitate an inductive proof of Theorem~\ref{thm:23_homomorphism_to_C9-restated}$'$, with cases (1)--(3) invoking the inductive hypothesis, and (4) corresponding to the base case;
see~\Cref{sec:proofs of two theorems}.

For distinct vertices $v,x,y\in V(G)$, to \emph{lift} a pair of edges $xv,vy\in E(G)$, we delete edges $xv$ and $vy$, and add a new edge $xy$. Note that lifting edge pairs may reduce a graph's edge-connectivity. We call $(xu,uy)$ an \emph{edge-pair}, $(xu,uv,vy)$ an \emph{edge-triple}, and $(xu,uv,vw,wy)$ an \emph{edge-quadruple}; in each case, all vertices are distinct. 
Throughout our proofs 
we frequently contract and lift edges.   Contracting always preserves planarity, as does lifting when the edges lie consecutively along a face boundary, 
which will always be the case in our proofs.

Now we can state our main result.

\begin{theorem}\label{thm:main_statement_SZ9_k8}
	Given a planar graph $G$, if $G\neq K_1$, $w(G)\geq0$, and $G/\mathcal{P}\notin \mathcal{N}$ for every partition $\mathcal{P}$, then at least one of the following statements holds.
	\begin{enumerate}[label=(\arabic*)]
		\item\label{thm:1_main_statement_SZ9_k8} $G$ contains an $\mathcal{N}$-good proper subgraph.
		\item\label{thm:2_main_statement_SZ9_k8} $G$ admits some lifting such that the resulting graph $G'$ contains an $\mathcal{N}$-good subgraph $H$ and $G'/H$ is $\mathcal{S}$-good.
        \item\label{thm:3_main_statement_SZ9_k8} $G$ admits some lifting at a vertex $v$ such that for the resulting graph $G''$, we know $G''-v$ is $\mathcal{S}$-good, and $G''/\mathcal{P}\notin \mathcal{N}$ for every partition $\mathcal{P}$ of $V(G'')$.  
		\item\label{thm:4_main_statement_SZ9_k8} $v(G)\leq 4$.
	\end{enumerate}
\end{theorem}

As we mentioned above, \Cref{thm:main_statement_SZ9_k8} gives a short proof of Theorem~\ref{thm:23_homomorphism_to_C9-restated}$'$. In addition, it also yields the following result on circular colorings of signed planar graphs.

\begin{theorem}\label{thm:23_circular_flow_signed_planar_graphs}
	For every signed planar graph $(G,\sigma)$ of girth at least $23$, there exists an $\varepsilon=\varepsilon(G,\sigma)$ such that $(G,\sigma)$ has a circular $(\frac{16}{7}-\varepsilon)$-coloring. 
\end{theorem}

In fact, we will prove the result in its dual version: for every $23$-edge-connected signed planar graph $(G,\sigma)$, there is an $\varepsilon=\varepsilon(G,\sigma)$ such that $(G,\sigma)$ admits a circular $(\frac{16}{7}-\varepsilon)$-flow.

In the next sections we present the definitions and notation needed to formally state and prove our remaining results. In~\Cref{subsec:strongly_Z9_connected}, we introduce a more general definition of orientations. 
This generality allows us to transform the problem of circular flow in planar graphs into one about group connectivity. 
In~\Cref{subsec:strongly_connected_32beta_orientations}, we introduce definitions and results about 
circular flows in signed planar graphs. 
In~\Cref{sec:proofs of two theorems} we use the content of~\Cref{subsec:strongly_Z9_connected} 
and \Cref{subsec:strongly_connected_32beta_orientations} to prove Theorem~\ref{thm:23_homomorphism_to_C9-restated}$'$ 
and \Cref{thm:23_circular_flow_signed_planar_graphs}, assuming the truth of \Cref{thm:main_statement_SZ9_k8}.
Finally, \Cref{sec:main theorem} is the heart of our work. There we prove \Cref{thm:main_statement_SZ9_k8}. 
We study the properties of a minimum counterexample and use the 
discharging method to obtain a contradiction, thus finishing the proof.

\section{Preliminaries}\label{sec:preliminary}

Throughout this paper we consider graphs with multiple edges but no loops. For a graph $G=(V,E)$, we write $v(G)$ and $e(G)$ for its numbers of vertices and edges. Let $\delta(G)$ be the minimum degree of $G$. For disjoint vertex subsets $X$ and $Y$ of $V(G)$, we write $[X,Y]$ for the set of edges in $G$ with one endpoint in $X$ and the other endpoint in $Y$, and let $e_{G}(X,Y):=|[X,Y]|$. Given a vertex subset $X$ of $G$, let $d(X):=|[X,X^c]|$. Further, let $\mu_{G}(uv)$ be the number of parallel edges between vertices $u$ and $v$. The
\emph{multiplicity} of $G$ is denoted by $\mu(G)$, where $\mu(G):=\max_{uv\in E(G)}\mu_G(uv)$. We use $kG$ to denote the graph formed by replacing every edge in $G$ with $k$ parallel edges. The \emph{odd-girth} of a graph $G$
is the length of its shortest odd cycle. Let $D$ be an orientation of $G$. An ordered pair $(u,v)$ denotes a directed 
edge $u\to v$. And $d_{D}^-(v)$ and $d_{D}^+(v)$ denote respectively the numbers of edges directed into and out of $v$. We call $D$ a \emph{strongly connected} orientation on $G$ if $d_D^+(S)>0$ and $d_D^-(S)>0$ for every proper subset $S\subset V(G)$.

Given a plane graph $G$, let $F(G)$ be the set of faces of $G$ and $f(G):=|F(G)|$. Given a face $f\in F(G)$, the \emph{degree} of $f$, denoted $d(f)$, is the number of edges with which $f$ is incident. A \emph{$k$-face} $f$ is a face with $d(f)=k$ and a \emph{$k^+$-face} $f$ is a face with $d(f)\geq k$. If two faces are incident with a common edge, then we call them \emph{adjacent}. Moreover, two faces $f$ and $f'$ are \emph{weakly adjacent} if there is a face chain $ff_1\cdots f_tf'$, where $f_i$ is a $2$-face for every $i\in\{1,\ldots,t\}$. 

Let $\mathcal{P}$ be a partition of $V(G)$ with parts $P_1,P_2,\ldots,P_t$ and $p_i:=|P_i|$ for each part. If $p_i\geq 2$ for some $i$, then we call $\mathcal{P}$ a \emph{nontrivial} partition; otherwise, $\mathcal{P}$ is \emph{trivial}. Throughout, we exclude the partition with a single part $\{V(G)\}$. In this work, we say a partition $\mathcal{P}$ has \emph{type $(k_1, k_2, \ast)$} (or type $(k_1^+, k_2^+, \ast)$) if $\mathcal{P}=\{P_1, P_2, \ldots, P_t\}$ satisfies that $|P_1|\geq |P_2|\geq \cdots \geq |P_t|$ and $|P_1|=k_1$ and $|P_2|=k_2$ (or $|P_1|\geq k_1$ and $|P_2|\geq k_2$, respectively). 

Let $G$ be a graph and $H$ be a connected subgraph of $G$. We contract $H$ to form a new vertex $w$. 
Let $\mathcal{P'}$ be a partition of $V(G/H)$ with parts $P_1, P_2, \ldots, P_t$. If $w\in P_1$, then we denote by $\mathcal{P}$ the partition \{$(P_1\setminus\{w\})\cup V(H), P_2, \ldots, P_t$\}; we call $\mathcal{P}$ the \emph{corresponding} partition of $V(G)$.

Given a graph $G=(V,E)$ and a mapping $\sigma: E(G)\to\{+1,-1\}$, the pair $(G,\sigma)$ is called a \emph{signed graph}. An edge is called \emph{positive} (or \emph{negative}) if $\sigma(e)=+1$ ($\sigma(e)=-1$, respectively).

\subsection{Strongly \texorpdfstring{$\mathbb{Z}_9$}{Z9}-connected graphs}\label{subsec:strongly_Z9_connected}

A modular $k$-orientation is a special case of the general notion of a $(\mathbb{Z}_{k},\beta)$-orientation, defined as follows. Given a mapping $\beta: V(G)\to \mathbb{Z}_k$, if $\sum_{v\in V(G)}\beta(v)\equiv 0 \pmod{k}$, then we call $\beta$ a 
\emph{$\mathbb{Z}_{k}$-boundary}. Given a $\mathbb{Z}_{k}$-boundary $\beta$ of $G$, a \emph{$(\mathbb{Z}_{k},\beta)$-orientation} is an orientation $D$ such that $d_D^+(v)-d_D^-(v)\equiv \beta(v)\pmod{k}$ for every vertex $v\in V(G)$. A graph $G$ is \emph{strongly $\mathbb{Z}_{k}$-connected} if for any $\mathbb{Z}_{k}$-boundary $\beta$, the graph $G$ has a $(\mathbb{Z}_{k},\beta)$-orientation. 
Clearly $\beta(v):=0$ for every $v$ is a $\mathbb{Z}_{k}$-boundary. Thus, every
strongly $\mathbb{Z}_9$-connected graph must have a modular $9$-orientation.

Given a positive odd integer $k$, we denote by $\mathcal{SZ}_{k}$ the family of all graphs that are strongly $\mathbb{Z}_k$-connected. Moreover, we have the following proposition about lifting and contracting for graphs in $\mathcal{SZ}_{k}$, which we will use frequently.

\begin{proposition}\label{prop:SZ9_lifting} 
    {\em{\cite{CL20}}}
	Let $G$ be a graph and $v\in V(G)$. Let $G'$ be a graph formed from $G$ by lifting some edge pairs at $v$ and $G''$ be formed from $G'$ by deleting the vertex $v$.
	\begin{enumerate}[label=(\roman*)]
		\item\label{prop:SZ9_contract} For any connected subgraph $H$ of $G$, if $H$ and $G/H$ belong to $\mathcal{SZ}_{k}$, then $G\in\mathcal{SZ}_{k}$.
		
		\item\label{prop:SZ9_lifting_1} For any connected subgraph $H$ of $G'$, if $H$ and $G'/H$ belong to $\mathcal{SZ}_{k}$, then $G\in\mathcal{SZ}_{k}$.
		
		\item \label{prop:SZ9_lifting_2} If $d_{G'}(v)\geq k-1$ and $G''\in\mathcal{SZ}_{k}$, then $G'\in\mathcal{SZ}_{k}$ and hence $G\in\mathcal{SZ}_{k}$.
        \item \label{prop:SZ9_contract_modular} For any connected subgraph $H$ of $G$, if $H\in\mathcal{SZ}_{k}$ and $G/H$ admits a modular $k$-orientation, then $G$ has a modular $k$-orientation.
	\end{enumerate}
\end{proposition}

All parts of the previous proposition have similar proofs. In each case we fix a desired boundary $\beta$ for $G$, start with the guaranteed orientation $D$ (with some corresponding boundaries) for the smaller graph ($G/H$, $G'/H$, or $G''$), and then extend $D$ to the desired 
$(\mathbb{Z}_{k},\beta)$-orientation for $G$ in the obvious way. So we omit the details.

We will often want to know that various graphs of small order are strongly $\mathbb{Z}_9$-connected.
For easy reference, we combine below some results from~\cite{LSWW23} and consequences of~\Cref{prop:SZ9_lifting}\ref{prop:SZ9_contract}.

\begin{lemma}\label{lem:SZ9_small_graphs}
    The following statements about graphs in $\mathcal{SZ}_9$ hold.
    \vspace{-.08in}
\begin{enumerate}[itemsep=-2pt, label=(\roman*)]
		\item\label{lem:4small_SZ_9} {\rm \cite{LSWW23}} The graphs in $\{\alpha K_2:\alpha\geq8\}\cup\{T_{a,b,c}:a+b+c\geq16,\delta\geq8\}\cup\{G\colon v(G)=4, e(G)\geq25,\delta(G)\geq8, \mu(G)\leq 7\}$ belong to $\mathcal{SZ}_{9}$. 

        \item Every multipath on three vertices with each edge's multiplicity at least $8$ is in $\mathcal{SZ}_{9}$.

        \item Let $G$ be a multigraph on four vertices containing $\alpha K_2$ where $\alpha\geq 8$. If $G/\alpha K_2$ has at least $16$ edges and $\delta\geq 8$, then $G$ is in $\mathcal{SZ}_{9}$.
\end{enumerate}
\end{lemma} 

For $(ii)$ and $(iii)$, we contract $8K_2$ which is in $\mathcal{SZ}_9$ and note that the resulting graph is still in $\mathcal{SZ}_9$, by \Cref{prop:SZ9_lifting}\ref{prop:SZ9_contract}, we are done.

\subsection{Strongly connected \texorpdfstring{$(32,\beta)$}{(32,B)}-orientations of graphs}\label{subsec:strongly_connected_32beta_orientations}

We also have an equivalent definition about strongly $\mathbb{Z}_{k}$-connected graphs, which is introduced in~\cite{LWZ20}.
\begin{definition}\label{def:(2k,beta)-boundary_orientation}
    \begin{enumerate}[label=(\roman*)]
        \item Given a graph $G$ and a mapping $\beta:V(G)\to\{0,\pm 1,\ldots,\pm k\}$, if $\beta(v)\equiv d(v)\pmod{2}$ for every vertex $v$ and $\sum_{v\in V(G)}\beta(v)\equiv 0\pmod{2k}$, then we call the mapping $\beta$ {\emph{a parity-compliant $\mathbb{Z}_{2k}$-boundary} (or \emph{$\mathbb{Z}_{2k}$-pc-boundary}, for short).}
        \item Given a {$\mathbb{Z}_{2k}$-pc-boundary} $\beta$ of $G$, if there exists an orientation $D$ such that $d_{D}^+(v)-d_{D}^-(v)\equiv\beta(v)\pmod{2k}$ for every vertex $v$, then we call $D$ a \emph{$(2k,\beta)$-orientation}.
    \end{enumerate}
\end{definition}

It has been proved in~\cite{LLL17} that a graph $G$ is strongly $\mathbb{Z}_{k}$-connected if and only if for every {$\mathbb{Z}_{2k}$-pc-boundary} $\beta$ of $G$, the graph $G$ admits a $(2k,\beta)$-orientation. Moreover, if a $(2k,\beta)$-orientation $D$ is strongly connected, then we call $D$ a \emph{strongly connected $(2k,\beta)$-orientation}. Using these definitions, we now define a new graph family, which is related to circular flows in signed graphs. 

\begin{definition}\label{def:strongly_connected_graph_family}
	Let $\mathcal{SC}_k$ be the family of all graphs that have a strongly connected $(2k,\beta)$-orientation for every {$\mathbb{Z}_{2k}$-pc-boundary} $\beta$.
\end{definition}

\begin{definition}\label{def_circular_flow}{\em \cite{LNWZ22}}
	Given positive integers $p$ and $q$ with $p$ being even, a \emph{circular $\frac{p}{q}$-flow} in a signed graph $(G,\sigma)$ is a pair $(D, f)$ where $D$ is an orientation on $G$ and $f: E(G)\to \mathbb{Z}$ satisfies the following conditions. 
	\begin{enumerate}[label=(\roman*)]
		\item For each positive edge $e$, $|f(e)|\in\{q,\ldots,p-q\}$;
		\item For each negative edge $e$, $|f(e)|\in\{0,\ldots,\frac{p}{2}-q\}\cup \{\frac{p}{2}+q,\ldots,p-1\}$;
		\item For each vertex $v$, we have $\sum\limits_{(v,w)\in D}f(vw)-\sum\limits_{(u,v)\in D}f(uv)\equiv 0\pmod{p}$.
	\end{enumerate}
 The \emph{circular flow index} of $(G,\sigma)$, denoted by $\phi_c(G,\sigma)$, is the minimum value $\frac{p}{q}$ such that $(G,\sigma)$  admits a circular $\frac{p}{q}$-flow. 
\end{definition}

For easy reference, let $\partial_{\!\scriptscriptstyle D}f(v):=\sum_{(v,w)\in D}f(vw)-\sum_{(u,v)\in D}f(uv)$. What follows is the strongest
known general result about circular flow value in a signed graph being less than $\frac{2k}{k-1}$.

\begin{theorem}\label{thm:circular_flow_results_in_signed_graph}
    {\em{\cite{LNWZ22}}} Given a signed graph $(G, \sigma)$ and an integer $k\geq 2$, if $G$ is $3k$-edge-connected, then $\phi_c(G,\sigma)<\frac{2k}{k-1}$.
\end{theorem}

As stated in \Cref{thm:23_circular_flow_signed_planar_graphs}, we improve \Cref{thm:circular_flow_results_in_signed_graph} when $(G,\sigma)$ is a signed planar graph and $k=8$. 
Specifically, we show that $\phi_c(G,\sigma)<\frac{16}{7}$ whenever $G$ is 23-edge-connected (rather than 24-edge-connected, as required by~\Cref{thm:circular_flow_results_in_signed_graph}). Next, we will show the connection between graphs in $\mathcal{SC}_{2k}$ and signed graphs having circular $r$-flow with $r<\frac{2k}{k-1}$. Before that, we need the following lemma.  

\begin{lemma}\label{lem:tight_cut}{\em{\cite{LNWZ22}}}
    Given a signed graph $(G,\sigma)$, if $\phi_c(G,\sigma)=r$, then every circular $r$-flow $(D,f)$ has a tight cut, where a tight cut $[X,X^c]$ of $(G,\sigma)$ with respect to $(D,f)$ is defined as follows: for every edge $uv\in E(G)$ with $u\in X$ and $v\in X^c$, 
    	\begin{equation*}
		f(uv)=\left\lbrace  
		\begin{array}{ll}
			1, &\text{if}~\sigma(uv)=+1 ~\text{and}~(u,v)\in D;\\
			r-1,&\text{if}~\sigma(uv)=+1 ~\text{and}~(v,u)\in D;\\
			\frac{r}{2}+1,&\text{if}~\sigma(uv)=-1 ~\text{and}~(u,v)\in D;\\
			\frac{r}{2}-1,&\text{if}~\sigma(uv)=-1 ~\text{and}~(v,u)\in D.
		\end{array}
		\right.
	\end{equation*}
\end{lemma}

Denote by $p^+(v)$ the number of positive edges incident with $v$ in $(G,\sigma)$ and recall that $2G$ is the graph formed by replacing every edge in $G$ with $2$ parallel edges.

\begin{theorem}\label{thm:main_flow_orientation}
	Given a signed graph $(G,\sigma)$ and a $\mathbb{Z}_{4k}$-pc-boundary $\beta$ such that $\beta(v)\equiv 2k\cdot p^+(v)\pmod{4k}$, if $2G$ admits a strongly connected $(4k,\beta)$-orientation, then $\phi_c(G,\sigma)<\frac{2k}{k-1}$.
\end{theorem}

\begin{proof}
	Let $D$ be a strongly connected $(4k,\beta)$-orientation on $2G$ with $\beta(v)\equiv 2k\cdot p^+(v)\pmod{4k}$. For every edge $e\in E(G)$, we denote by $\{e_1,e_2\}$ the parallel edges in $E(2G)$ corresponding to $e$. Let $D_1$ be an auxiliary orientation on $G$. We first define a mapping $f_1$ on $E(G)$ such that $f_1(e)=I(e_1)+I(e_2)$, where $I$ is an indicator function on $E(2G)$ such that $I(e_i)=1$ if $D(e_i)=D_1(e)$ and $I(e_i)=-1$ otherwise. It is easy to see that
$f_1(e)\in\{-2,0,2\}$ and $\partial_{\!\scriptscriptstyle D_1} f_1(v)=d_{D}^+(v)-d_{D}^-(v)\equiv \beta(v)\pmod{4k}$ for every vertex $v\in V(G)$. We then define another mapping $f_2:E(G,\sigma)\to \{0,2k\}$ such that $f_2(e)=2k$ if $e$ is positive and $f_2(e)=0$ otherwise. Thus $\partial_{\!\scriptscriptstyle D_1}f_2(v)\equiv2k\cdot p^+(v)\pmod{4k}$ for each vertex $v\in V(G)$. Let $f:=f_1+f_2$. We know that in $(G,\sigma)$, $f(e)\in \{2k-2,2k,2k+2\}$ for every positive edge $e$ and $f(e)\in\{-2,0,2\}$ for every negative edge $e$. Moreover, $\partial_{\!\scriptscriptstyle D_1} f(v)=\partial_{\!\scriptscriptstyle D_1} f_1(v)+\partial_{\!\scriptscriptstyle D_1} f_2(v)\equiv \beta(v)+2k\cdot p^+(v)\equiv0\pmod{4k}$. Thus, $(D_1,f)$ is a circular $\frac{4k}{2k-2}$-flow in $(G,\sigma)$, which implies $\phi_c(G,\sigma)\leq \frac{4k}{2k-2}=\frac{2k}{k-1}$.
	
	To show that $\phi_c(G,\sigma)<\frac{2k}{k-1}$, suppose to the contrary that $\phi_c(G,\sigma)=\frac{2k}{k-1}$. By~\Cref{lem:tight_cut}, every circular $\frac{2k}{k-1}$-flow of $(G,\sigma)$ has a tight cut. In particular, there is a tight cut $[X,X^c]$ with respect to $(D_1,f)$. By the definition of tight cuts, for every edge $e=uv\in[X,X^c]$ with $u\in X$ and $v\in X^c$, when $\sigma(e)=+1$, $f(e)=2k-2$ if $(u,v)\in D_1$ and $f(e)=2k+2$ if $(v,u)\in D_1$; when $\sigma(e)=-1$, $f(e)=-2$ if
$(u,v)\in D_1$ and $f(e)=2$ if $(v,u)\in D_1$. By the definitions of mappings $I$ and $f$, if $f(e)=2k+2$ or $f(e)=2$, then $e_1$ and $e_2$ have the same direction in $D$, and $e$ is oriented in $D_1$ the same as the $e_i$'s in $D$; if $f(e)=2k-2$ or $f(e)=-2$, then $e_1$ and $e_2$ have the same direction in $D$, and $e$ is oriented in $D_1$ opposite to the $e_i$'s in $D$. Therefore, under the orientation $D$, all the edges of $E(2G)$ are oriented from $X^c$ to $X$; thus $D$ is not strongly
connected (as $d_{D}^+(X)=0$), which is a contradiction. Hence $\phi_c(G,\sigma)<\frac{2k}{k-1}$.
\end{proof}

The conclusion of~\Cref{thm:23_circular_flow_signed_planar_graphs} is the conclusion of~\Cref{thm:main_flow_orientation}, 
with $k=8$. So it suffices to prove that $2G\in\mathcal{SC}_{16}$ for every $23$-edge-connected planar graph $G$. Hence, the graph family $\mathcal{SC}_{k}$ plays an important role. 

Given a graph $G$ and its connected subgraph $H$, let $G':=G/H$ and let $w$ denote the new vertex formed by contracting $H$. 
For every {$\mathbb{Z}_{2k}$-pc-boundary} $\beta$ of $G$, a mapping $\beta': V(G')\to \{0,\pm1,\ldots,\pm k\}$ is defined 
as follows: $\beta'(w):\equiv\sum_{v\in V(H)}\beta(v)\pmod{2k}$ and 
$\beta'(v):=\beta(v)$ 
for every $v\in V(G')\backslash \{w\}$. 
Note that $\beta'$ is a {$\mathbb{Z}_{2k}$-pc-boundary} of $G'$.

Fix a graph $G$ and its subgraph $H$, a {$\mathbb{Z}_{2k}$-pc-boundary} $\beta$ of $G$, and $\beta'$ of $G'$ as above.
Note that $G$ admits a $(2k,\beta)$-orientation if $G'$ admits a $(2k,\beta')$-orientation and $H$ admits a $(2k,\beta'')$-orientation 
for every $\mathbb{Z}_{2k}$-pc-boundary $\beta''$ of $H$.
And if these orientations for $G'$ and $H$ are both strongly connected, then so is the orientation for $G$. 
Thus, if $H\in \mathcal{SC}_{k}$ and $G/H\in\mathcal{SC}_{k}$, then also $G\in\mathcal{SC}_{k}$. To push this idea further, we seek the largest class of graphs $H$ to which it applies. This motivates the following two definitions, similar to those in~\cite{HLLW20}.

\begin{definition}\label{def:nice_supergraph}
	Let $G$ be a graph. If $H\subseteq G$ and there exist distinct vertices $x,y\in V(H)$ such that we have an $(x,y)$-path in $G-E(H)$, then we call $G$ a \emph{nice supergraph} of $H$.
\end{definition}

\begin{definition}\label{def:weakly_contractible}
	Let $k$ be a positive integer. Given a graph $H$ and its nice supergraph $G$, we call $H$ \emph{weakly contractible} if for every {$\mathbb{Z}_{2k}$-pc-boundary} $\beta$ of $G$, 
    every strongly connected $(2k,\beta')$-orientation of $G/H$ can be extended to a strongly connected $(2k,\beta)$-orientation 
    of $G$. 
    
    Let $\mathcal{W}_{k}$ be the family of all graphs that are weakly contractible. 
\end{definition}
	
From the definitions, it is straightforward to check that $\mathcal{SC}_{k}\subseteq\mathcal{W}_{k}$. Similar to the connection between the graph families $\mathcal{W}_{3}$ and $\mathcal{SC}_{3}$ shown in \cite{HLLW20}, we have the following proposition about the graph families $\mathcal{W}_{k}$ and $\mathcal{SC}_{k}$. 
	
\begin{proposition}\label{prop:H_Wk_SCk}
	A graph $H\in\mathcal{W}_{k}$ if and only if $H+xy\in \mathcal{SC}_k$ for all distinct $x,y\in V(H)$.
\end{proposition}

Note that we allow graphs to have multiple edges, so possibly $xy\in e(H)$. 
As an immediate corollary of the above proposition, we have the following.

\begin{proposition}\label{prop:SC_SZ_W_graphs}
	\begin{enumerate}[itemsep=-2pt, label=(\arabic*)]
		\item\label{prop:SC_SZ_W_graph_1_SCk} If $G$ contains a Hamiltonian cycle $C$ such that $G-E(C)\in\mathcal{SZ}_{k}$, then $G\in\mathcal{SC}_k$.
		\item\label{prop:SC_SZ_W_graph_2_Wk} If for any two distinct vertices $x,y$ of a graph $G$, there is a Hamiltonian path $P_{xy}$ such that $G-E(P_{xy})\in\mathcal{SZ}_{k}$, then $G\in\mathcal{W}_k$.
	\end{enumerate}
\end{proposition}
\begin{proof}
(1) If $D$ is a $(2k,\beta)$-orientation, for some given boundary $\beta$, then adding a cyclicly oriented Hamiltonian cycle yields a strongly connected digraph that is also a $(2k,\beta)$-orientation.  (2) follows directly from (1) by~\Cref{prop:H_Wk_SCk}.
\end{proof}
For our proofs in the following sections, we will also make use of the next proposition.  Both its statement and its proof are
analogous to those of \Cref{prop:SZ9_lifting}.

\begin{proposition}\label{prop:lifting_WSC16}
	Given a graph $G$ and a vertex $v\in V(G)$, let $G'$ be a graph formed from $G$ by lifting some edge pairs at $v$ and let $G''$ be a graph formed from $G'$ by deleting the vertex $v$. Then the following statements hold.
	\begin{enumerate}[label=(\roman*)]
 
		\item \label{prop:WSC16_contract} Given a connected subgraph $H$ of $G$, if $H\in \mathcal{W}_k$ and $G/H\in\mathcal{SC}_k$, then $G\in\mathcal{SC}_k$.
		
		\item \label{prop:WSC16_lifting_1} Given a connected subgraph $H$ of $G'$, if $H\in \mathcal{W}_k$ and $G'/H\in\mathcal{SC}_k$, then $G\in\mathcal{SC}_k$.
		
		\item \label{prop:WSC16_lifting_2} If $d_{G'}(v)\geqslant k$ and $G''\in\mathcal{SC}_k$, then $G\in\mathcal{SC}_k$.
	\end{enumerate}
\end{proposition}

In~\cite{LSWW23}, all graphs in $\{15 K_2\}\cup\{T_{a,b,c}:a+b+c=30,\delta\geq15\}\cup\{G\colon v(G)=4, e(G)=46,\delta(G)\geq15, \mu(G)\leq 15\}$ were shown to belong to $\mathcal{SZ}_{16}$. By applying~\Cref{prop:SC_SZ_W_graphs} and \Cref{prop:lifting_WSC16}\ref{prop:WSC16_contract}, we have the following results. 

\begin{lemma}\label{lem:WSC_small_graphs}
	\begin{enumerate}[itemsep=-2pt, label=(\roman*)]
		\item\label{lem:W16_small_graphs} The graphs in $\{16 K_2\}\cup\{T_{a,b,c}:a+b+c=32,\delta\geq17\}\cup\{G\colon v(G)=4, e(G)=49,\delta(G)\geq17, \mu(G)\leq 16\}$ belong to $\mathcal{W}_{16}$.   
		\item \label{lem:SC16_small_graphs} The graphs in $\{\alpha K_2:\alpha\geq17\}\cup\{T_{a,b,c}:a+b+c\geq33,\delta\geq17\}\cup\{G\colon v(G)=4, e(G)\geq50,\delta(G)\geq17, \mu(G)\leq 16\}$ belong to $\mathcal{SC}_{16}$.
        \item\label{lem:path_SC16} Every multipath on three vertices with each edge's multiplicity at least $17$ is in $\mathcal{SC}_{16}$.
        \item\label{lem:four_graphs_alpha K2} Let $G$ be a multigraph on four vertices containing $\alpha K_2$ where $\alpha\geq 17$. If $G/\alpha K_2$ has at least $33$ edges and $\delta\geq 17$, then $G$ is in $\mathcal{SC}_{16}$.
	\end{enumerate}
\end{lemma} 

\begin{proof}
    For graphs in~$\ref{lem:W16_small_graphs}$, we delete an arbitrary Hamiltonian path; for graphs in~$\ref{lem:SC16_small_graphs}$, we delete an arbitrary Hamiltonian cycle. Note that each of the resulting graphs 
is in $\{15 K_2\}\cup\{T_{a,b,c}:a+b+c=30,\delta\geq15\}\cup\{G\colon v(G)=4, e(G)=46,\delta(G)\geq15, \mu(G)\leq 15\}\subset \mathcal{SZ}_{16}$.  So \ref{lem:W16_small_graphs} and \ref{lem:SC16_small_graphs} hold by \Cref{prop:SC_SZ_W_graphs}. 
    
    For~$\ref{lem:path_SC16}$, let $P$ be a multipath with vertex set $\{x,y,z\}$ and $\mu(xy)\geq 17$ and $\mu(yz)\geq 17$. By~$\ref{lem:SC16_small_graphs}$, the two-vertex subgraph $P_{xy}$ induced by $x$ and $y$ is in $\mathcal{SC}_{16}(\subset \mathcal{W}_{16})$ and, moreover,  $P/P_{xy}\in\mathcal{SC}_{16}$. By~\Cref{prop:lifting_WSC16}\ref{prop:WSC16_contract}, $P\in\mathcal{SC}_{16}$.
    
    For~$\ref{lem:four_graphs_alpha K2}$, note that $G$ contains a subgraph $\alpha K_2$ which belongs to  $\mathcal{SC}_{16}$. 
Moreover, $G/\alpha K_2$ has three vertices and at least $33$ edges, and $\delta(G/\alpha K_2)\geq 17$. By~$\ref{lem:SC16_small_graphs}$ and $\ref{lem:path_SC16}$, $G/\alpha K_2$ also belongs to $\mathcal{SC}_{16}$. By~\Cref{prop:lifting_WSC16}\ref{prop:WSC16_contract}, $G\in\mathcal{SC}_{16}$.
\end{proof}

\section[two applications]{Proofs of Theorem~\ref{thm:23_homomorphism_to_C9-restated}$'$ and  Theorem~\ref{thm:23_circular_flow_signed_planar_graphs}}\label{sec:proofs of two theorems}
In this section, we assume the truth of~\Cref{thm:main_statement_SZ9_k8} and use it to prove
Theorem~\ref{thm:23_homomorphism_to_C9-restated}$'$ and \Cref{thm:23_circular_flow_signed_planar_graphs}.
We remark that every graph $G$ that is $\mathcal{N}$-good or $\mathcal{S}$-good satisfies the hypotheses of \Cref{thm:main_statement_SZ9_k8}; thus, such a graph $G$ also satisfies its conclusion. 

Now assuming that~\Cref{thm:main_statement_SZ9_k8} is true, we first obtain~\Cref{thm:partition_SZ9_SC16} below. Part (2) immediately implies (the flow version of)~\Cref{thm:23_circular_flow_signed_planar_graphs}; part (1) implies
a slight weakening of~Theorem~\ref{thm:23_homomorphism_to_C9-restated}$'$, which requires edge-connectivity at least $23$, rather than odd-edge-connectivity at least $23$.  Proving the final version of that result requires one more trick, which we present at the end of this section.

\begin{theorem}\label{thm:partition_SZ9_SC16}
	Let $G$ be a planar graph such that $G\neq K_1$ and $w(G)\geq 0$. 
	\begin{enumerate}[label=(\arabic*)]
		\item\label{thm:partition_SZ9} If $G/\mathcal{P}\not\in \mathcal{N}$ for every partition $\mathcal{P}$, then $G\in \mathcal{SZ}_9$.
        \item\label{thm:partition_SC16} If $G/\mathcal{P}\notin \mathcal{W}^*\cup\mathcal{N}$ for every partition $\mathcal{P}$ (that is, $G$ is $\mathcal{S}$-good), then  $2G\in \mathcal{SC}_{16}$. 
	\end{enumerate}
\end{theorem}

This result is sharp in the following sense.  For each $G\in \mathcal{N}$, it is straightforward to construct a boundary showing
that $G\notin \mathcal{SZ}_9$.  Similarly, for each $G\in \mathcal{W}^*\cup\mathcal{N}$ it is straightforward to construct a 
boundary showing that $2G\notin \mathcal{SC}_{16}$. 

\begin{proof}
\ref{thm:partition_SZ9} Assume to the contrary that $G$ is a minimum counterexample with $v(G)+e(G)$ being minimized. Note that $G$ satisfies all the hypotheses of \Cref{thm:main_statement_SZ9_k8} and thus its conclusion must hold for $G$. 
We consider the four cases. 
\begin{itemize}[itemsep=-2pt]
    \item If \Cref{thm:main_statement_SZ9_k8}\ref{thm:1_main_statement_SZ9_k8} holds, then $G$ contains an $\mathcal{N}$-good proper subgraph $H$. By the definition of $\mathcal{N}$-good graphs, $H/\mathcal{P}\notin\mathcal{N}$ for every partition $\mathcal{P}$ of $V(H)$. As $H$ is a proper subgraph of $G$, by the minimality of $G$, the graph $H\in\mathcal{SZ}_{9}$. Furthermore, note that $w(G/H)\geq 0$, that $G/H$ has no partition $\mathcal{P}$ such that $(G/H)/\mathcal{P}\in\mathcal{N}$, and that $G/H$ has fewer vertices than $G$ (as $v(H)\geq 2$). Again since $G$ is a minimal counterexample, $G/H\in\mathcal{SZ}_{9}$. By~\Cref{prop:SZ9_lifting}\ref{prop:SZ9_contract}, $G\in\mathcal{SZ}_{9}$, a contradiction. 
    
    \item If~\Cref{thm:main_statement_SZ9_k8}\ref{thm:2_main_statement_SZ9_k8} holds, then the resulting graph $G'$ formed by some lifting contains an $\mathcal{N}$-good subgraph $H$ and $G'/H$ is $\mathcal{S}$-good. It is easily observed that each of $H$ and $G'/H$ satisfies the conditions of this theorem but has fewer edges than $G$. Hence both $H$ and $G'/H$ belong to $\mathcal{SZ}_{9}$. By~\Cref{prop:SZ9_lifting}\ref{prop:SZ9_lifting_1}, $G\in\mathcal{SZ}_{9}$, a contradiction. 
    
    \item If~\Cref{thm:main_statement_SZ9_k8}\ref{thm:3_main_statement_SZ9_k8} holds, then the resulting graph $G''$ formed from $G$ by lifting at a vertex $v$ satisfies that $G''-v$ is $\mathcal{S}$-good, and also that $G''/\mathcal{P}\notin \mathcal{N}$ for every partition $\mathcal{P}$ of $V(G'')$. We first claim that $G''-v\in\mathcal{SZ}_{9}$ because it satisfies the conditions of this theorem (as it is $\mathcal{S}$-good) but has fewer vertices than $G$. Since $G''/\mathcal{P}\notin\mathcal{N}$ for every partition $\mathcal{P}$ of $V(G'')$, we know that $d_{G''}(v)\geq 8$. By~\Cref{prop:SZ9_lifting}\ref{prop:SZ9_lifting_2}, $G\in\mathcal{SZ}_{9}$, which is a contradiction.
    
    \item If~\Cref{thm:main_statement_SZ9_k8}\ref{thm:4_main_statement_SZ9_k8} holds, then we assume that $2\leq v(G)\leq 4$. We will describe $G$ explicitly and use~\Cref{lem:SZ9_small_graphs} to show that $G\in\mathcal{SZ}_{9}$. That $G/\mathcal{P}\notin \{\alpha K_2: \alpha\leq 7\}\subset \mathcal{N}$ for every partition $\mathcal{P}$ implies that $\delta(G)\geq 8$. If $v(G)=2$, then $G$ is $\alpha K_2$ with $\alpha\geq 8$. If $v(G)=3$, then $G$ is either a multipath with $\delta\geq 8$ or a multitriangle $T_{a,b,c}$ with $a+b+c\geq 16$ as $G/\mathcal{P}\notin\{T_{a,b,c},a+b+c\leq 15\}$ for every partition $\mathcal{P}$. If $v(G)=4$, then $e(G)\geq 25$ since $w(G)\geq 0$. Furthermore, if $G$ contains $\alpha K_2$ with $\alpha\geq 8$, then $G/\alpha K_2$ has at least $16$ edges since $G/\alpha K_2 \notin\{T_{a,b,c},a+b+c\leq 15\}$ and $G/\alpha K_2$ has minimum degree at least $8$; if $\mu(G)\leq 7$, then $G$ satisfies that $e(G)\geq 25$, $\mu(G)\leq 7$, and $\delta(G)\geq 8$. According to~\Cref{lem:SZ9_small_graphs}, each of the graphs $G$ described above belongs to $\mathcal{SZ}_{9}$.
\end{itemize}
		
\ref{thm:partition_SC16} Let $G$ be a counterexample minimizing $v(G)+e(G)$. It means  $2G\notin\mathcal{SC}_{16}$. Note that $G$ is $\mathcal{S}$-good and thus the conclusion of \Cref{thm:main_statement_SZ9_k8} holds for $G$. Similarly we consider four cases, each leading to the contradiction that $2G\in\mathcal{SC}_{16}$.
\begin{itemize}[itemsep=-2pt]
    \item If~\Cref{thm:main_statement_SZ9_k8}\ref{thm:1_main_statement_SZ9_k8} holds, then $G$ contains an $\mathcal{N}$-good proper subgraph $H$. By hypothesis, $w(G/H)\geq w(G)\geq 0$; and also $(G/H)/\mathcal{P}\notin\mathcal{N}\cup\mathcal{W}^*$ for every partition $\mathcal{P}$. Thus, $G/H$ is $\mathcal{S}$-good. So $(2G)/(2H)\in\mathcal{SC}_{16}$, by the minimality of $G$. If $H/\mathcal{P}\notin\mathcal{W}^*$ for every partition $\mathcal{P}$, then $H$ is $\mathcal{S}$-good; thus the minimality of $G$ gives $2H\in\mathcal{SC}_{16}\subseteq \mathcal{W}_{16}$. If $H/\mathcal{P}_0\in \mathcal{W}^*$ for the trivial partition $\mathcal{P}_0$, then $H\in\mathcal{W}^*$. In this case $2H\in \{16K_2\}\cup\{T_{a,b,c}:a+b+c=32,\delta\geq 18\}$ and $2H\in\mathcal{W}_{16}$ by \Cref{lem:WSC_small_graphs}\ref{lem:W16_small_graphs}. Now by~\Cref{prop:lifting_WSC16}\ref{prop:WSC16_contract} we get  $2G\in\mathcal{SC}_{16}$, a contradiction.
    
    \item If \Cref{thm:main_statement_SZ9_k8}\ref{thm:2_main_statement_SZ9_k8} holds, then $G'/H$ is $\mathcal{S}$-good; the minimality of $G$ gives $(2G')/(2H)\in\mathcal{SC}_{16}\subseteq \mathcal{W}_{16}$. Similarly, if $H/\mathcal{P}\notin\mathcal{W}^*$ for every partition $\mathcal{P}$, then $2H\in\mathcal{SC}_{16}$. But if $H/\mathcal{P}_0\in \mathcal{W}^*$ for the trivial partition $\mathcal{P}_0$, then $H\in\mathcal{W}^*$ by definition; as above, $2H\in\mathcal{W}_{16}$. In both cases, by \Cref{prop:lifting_WSC16}\ref{prop:WSC16_lifting_1} we get $2G\in\mathcal{SC}_{16}$, a contradiction.

    \item If \Cref{thm:main_statement_SZ9_k8}\ref{thm:3_main_statement_SZ9_k8} holds, then $G''-v$ is $\mathcal{S}$-good. Hence $2(G''-v)\in\mathcal{SC}_{16}$, by the minimality of $G$. Since $G''/\mathcal{P}\notin\mathcal{N}$ for every partition $\mathcal{P}$, we have $d_{2G''}(v)\geq 16$. Thus by~\Cref{prop:lifting_WSC16}\ref{prop:WSC16_lifting_2} we get $2G\in\mathcal{SC}_{16}$, a contradiction.
    
    \item If \Cref{thm:main_statement_SZ9_k8}\ref{thm:4_main_statement_SZ9_k8} holds, then $2\leq v(G)\leq 4$. We will describe $G$ explicitly and use \Cref{lem:WSC_small_graphs} to show that $2G\in \mathcal{SC}_{16}$. Since $G/\mathcal{P}\notin \mathcal{W}^*\cup\mathcal{N}$ for every partition $\mathcal{P}$, $\delta(G)\geq 9$. If $v(G)=2$, then $2G=\alpha K_2$ with $\alpha\geq 18$. If $v(G)=3$, then $2G$ is either a multipath with $\delta(2G)\geq 18$ or a multitriangle $T_{a,b,c}$ with $a+b+c\geq 34$ since $G/\mathcal{P}\notin\{T_{a,b,c},a+b+c\leq 16\}$ for every partition $\mathcal{P}$. If $v(G)=4$, then $e(2G)=2e(G)\geq 50$ since $w(G)\geq 0$. If $2G$ has a subgraph $\alpha K_2$ with $\alpha\geq 17$, then $2G/\alpha K_2$ contains at least $34$ edges because $G$ has no partition $\mathcal{P}$ such that $G/\mathcal{P}\in\{T_{a,b,c},a+b+c\leq 16\}$ and, furthermore, such $2G/\alpha K_2$ has minimum degree at least $18$. If $\mu(2G)\leq 16$, then $2G$ satisfies that $e(2G)\geq 50$, $\delta(2G)\geq 18$ and $\mu(2G)\leq 16$. By \Cref{lem:WSC_small_graphs}, in each case $2G\in \mathcal{SC}_{16}$.
\end{itemize}
This completes the proof of the theorem.
\end{proof}

\begin{corollary}\label{cor:23_circula_flow}
    If $G$ is a 23-edge-connected planar graph, then:
    \vspace{-.08in}
	\begin{enumerate}[itemsep=-2pt, label=(\arabic*)]
		\item \label{cor:23_circular_9/4} $G\in\mathcal{SZ}_{9}$; and
		 
		\item \label{cor:23_circular_16/7} $2G\in \mathcal{SC}_{16}$. 
	\end{enumerate}
\end{corollary}

\begin{proof}
	For~\ref{cor:23_circular_9/4}, since $G$ is $23$-edge-connected, for every partition $\mathcal{P}$ with parts $P_1,P_2,\ldots,P_t$, we have $d_{G}(P_i)\geq 23$ for all $i\in [t]$. Thus $w(G)\geq 23t-23t+42=42$, which implies that $G/\mathcal{P}\notin\mathcal{N}$ for every partition $\mathcal{P}$ by~\Cref{ob:partition_bound_notin}\ref{ob:3_12_N}. By \Cref{thm:partition_SZ9_SC16}\ref{thm:partition_SZ9}, $G\in\mathcal{SZ}_{9}$.

    For~\ref{cor:23_circular_16/7}, since $G$ is $23$-edge-connected, again $w(G)\geq 42$ and $G/\mathcal{P}\notin\mathcal{W}^*\cup\mathcal{N}$ for every partition $\mathcal{P}$ by~\Cref{ob:partition_bound_notin}\ref{ob:1_14_WN}. It follows from \Cref{thm:partition_SZ9_SC16}\ref{thm:partition_SC16} that $2G\in\mathcal{SC}_{16}$. 
\end{proof}

\Cref{thm:23_circular_flow_signed_planar_graphs} is a straightforward consequence of the above result.  \Cref{cor:23_circula_flow}\ref{cor:23_circular_16/7} implies that $2G$ admits a strongly connected $(32,\beta)$-orientation for every given {$\mathbb{Z}_{32}$-pc-boundary} $\beta$ with $\beta(v)\equiv 16\cdot p^+(v)\pmod{32}$. By \Cref{thm:main_flow_orientation}, 
for every signature $\sigma$ of $G$ 
we get $\phi_c(G,\sigma)<\frac{16}{7}$.

To obtain Theorem~\ref{thm:23_homomorphism_to_C9-restated}$'$, we need Zhang's splitting lemma~\cite{Z02} below. 

\begin{lemma}\label{lem:splitting}\emph{\cite{Z02}}
Let $G=(V,E)$ be a graph with odd-edge-connectivity $\lambda$. Assume there is a vertex $v\in V(G)$ such that $d(v)\notin\{2,\lambda\}$. Arbitrarily label the edges of $G$ incident with $v$ as $\{e_1,e_2,\ldots,e_{d(v)}\}$. Then there is an integer $i\in \{1,2,\ldots,d(v)\}$ such that the new graph formed from $G$ by splitting $e_i$ and $e_{i+1}$ (indices $i, i+1$ are taken modulo $d(v)$) away from $v$ remains of odd-edge-connectivity $\lambda$.
\end{lemma}

\smallskip
\noindent
\emph{Proof of Theorem~\ref{thm:23_homomorphism_to_C9-restated}\,$'$}. We will prove that every odd-$23$-edge-connected planar graph admits a modular $9$-orientation. Assume to the contrary that $G$ is a counterexample minimizing $v(G)$. If $\delta(G)<23$, then there exists $v$ with $d(v)<23$ and $d(v)$ even.  We label the edges incident to $v$ in their (cyclic) order around $v$ in the plane embedding of $G$.  By \Cref{lem:splitting}, we can lift off all of these edges from $v$ (in pairs) 
such that the resulting graph $G'$ is again odd-23-edge-connected.  Thus, $G'$ is a smaller counterexample, contradicting our choice of $G$.  Hence, we assume that $\delta(G)\geq 23$. 

If $G$ is $23$-edge-connected, then \Cref{cor:23_circula_flow}\ref{cor:23_circular_9/4} implies that $G\in\mathcal{SZ}_{9}$; in particular, $G$ admits a modular $9$-orientation. Thus we assume that $[X, X^c]$ is an edge-cut where $d(X)< 23$ and $|X|$ is minimized. Note that $|X|\geq 2$ and for every subset $X'\subsetneq X$, $d(X')\geq 23$. Let $H:=G[X]$. By the minimality of $X$, $H$ is connected. For every partition $\mathcal{Q}$ of $V(H)$ with parts $Q_1,Q_2,\ldots,Q_t$, we have $w_H(\mathcal{Q})= \sum_{i=1}^{t}d_H(Q_i)-23t+42=(\sum_{i=1}^{t}d_G(Q_i)-d_G(X))-23t+42\geq 23t-23-23t+42=19$. Thus $w(H)\geq 19$, which implies that $H/\mathcal{P}\notin\mathcal{N}$ for every partition $\mathcal{P}$ by \Cref{ob:partition_bound_notin}\ref{ob:3_12_N}. By \Cref{thm:partition_SZ9_SC16}\ref{thm:partition_SZ9}, $H\in\mathcal{SZ}_{9}$. As $G$ is a minimum counterexample and $v(H)=|X|\geq 2$, $G/H$ admits a modular $9$-orientation. Hence, $G$ has a modular $9$-orientation by \Cref{prop:SZ9_lifting}\ref{prop:SZ9_contract_modular}, a contradiction. \qed

\section{Main theorem}\label{sec:main theorem}

This section is devoted to proving \Cref{thm:main_statement_SZ9_k8}. 
Below, we recall its statement. 
But first we remind the reader 
of the definition of the weight function, given in~\Cref{equ:weight_function}, as 
well as the notions of $\mathcal{N}$-good and $\mathcal{S}$-good from \Cref{def:graph_family}.

\smallskip
\noindent
\textbf{{\Cref*{thm:main_statement_SZ9_k8}}.} \textit{Given a planar graph $G$, if $G\neq K_1$, $w(G)\geq0$, and $G/\mathcal{P}\notin \mathcal{N}$ for any partition $\mathcal{P}$, then at least one of the following 4 statements holds.
	\begin{enumerate}[itemsep=-2pt, label=(\arabic*)]
		\item $G$ contains an $\mathcal{N}$-good proper subgraph.
		\item $G$ admits some lifting such that the resulting graph $G'$ contains an $\mathcal{N}$-good subgraph $H$ and $G'/H$ is $\mathcal{S}$-good.
        \item $G$ admits some lifting at a vertex $v$ such that for the resulting graph $G''$ we know $G''-v$ is 
            $\mathcal{S}$-good, and $G''/\mathcal{P}\notin \mathcal{N}$ for every partition $\mathcal{P}$ of $V(G'')$.  
		\item $v(G)\leq 4$.
\end{enumerate}}

In the rest of the proof we will frequently be reasoning about graphs that are $\mathcal{N}$-good or $\mathcal{S}$-good. So it
is helpful to recall the following observation.

\noindent
\textbf{{\Cref*{ob:W*_N_good}}.} 
\textit{
    Every $\mathcal{S}$-good graph is also $\mathcal{N}$-good. Every graph in $\mathcal{W}^*$ is $\mathcal{N}$-good (but not $\mathcal{S}$-good).}

We also introduce another $\mathcal{N}$-good graph: $6C_4^+$, which is used in the next subsection. We write $6C_4^+$ to denote the graph formed from $6C_4$ by adding one more parallel edge between two adjacent vertices. 
Clearly, $v(6C_4^+)=4$, $e(6C_4^+)=25$, and $\delta(6C_4^+)=12$. 
Furthermore, by~\Cref{ob:weight_small_graphs}, $w(6C_4^+)=2e(6C_4^+)-50=0$. Obviously $6C_4^+\notin \mathcal{N}$. For every nontrivial partition
$\mathcal{P}$ of $6C_4^+$, we have $6C_4^+/\mathcal{P}\in\{\alpha K_2:\alpha \ge 12\}\cup\{T_{a,b,c}:a+b+c\ge 18, \delta\ge 12\}$.
Thus, $6C_4^+$ is an $\mathcal{N}$-good graph.

To prove~\Cref{thm:main_statement_SZ9_k8}, we assume the result is false and let $G$ be a counterexample minimizing $v(G)+e(G)$. 
Note that $v(G)\geq 5$ and $G$ contains no $\mathcal{N}$-good proper subgraph. In the next subsection, we prove our lower bound on the edge connectivity of this minimal counterexample, and find some configurations forbidden in $G$.  Finally, in \Cref{sec:discharging} 
we use the discharging method to complete the proof of \Cref{thm:main_statement_SZ9_k8}. 

\subsection{Forbidden configurations}

Recall that our minimal counterexample $G$ satisfies $w(G)\geq 0$. Next we prove a useful lemma, called the 
\emph{gap lemma}, indicating that for any nontrivial partition $\mathcal{P}$ of $V(G)$ the weight value $w_{G}(\mathcal{P})$ is much larger than $0$. The gap lemma allows us to lift edge pairs and guarantee that the resulting graph still has a non-negative weight. 
We use this lemma frequently in the rest of the proof.   We recall that a partition $\mathcal{P}$ \emph{has 
type $(p_1^+,p_2^+,*)$} when $\mathcal{P}$ has parts $P_1,P_2,\ldots$ with $|P_1|\ge p_1$, $|P_2|\ge p_2$, and 
all other parts (if they exist) have no requirement on their sizes.

\begin{lemma}[Gap Lemma]\label{lem:partition_type_value}
	Let $\mathcal{P}$ be a nontrivial partition  of $V(G)$ with parts $P_1,P_2,\ldots,P_t$ such that $|P_1|\geq |P_2|\geq \cdots \geq |P_t|$.
\begin{enumerate}[label=(\alph*)]
	\item \label{lem:1_partition_type_value_2111} 
 if $\mathcal{P}$ has type $(2^+,1^+,*)$, then $w_G(\mathcal{P})\geq 9$; 
 
	\item \label{lem:2_partition_type_value_3111} 
 if $\mathcal{P}$ has type $(3^+,1^+,*)$, then $w_G(\mathcal{P})\geq 16$; 
 
	\item \label{lem:3_partition_type_value_2211} 
 if $\mathcal{P}$ has type $(2^+,2^+,*)$, then $w_G(\mathcal{P})\geq 18$; 
 
	\item \label{lem:4_partition_type_value_4111} 
 if $\mathcal{P}$ has type $(4^+,1^+,*)$, then $w_G(\mathcal{P})\geq 20$;  

    \item \label{lem:5_partition_type_value_3211} 
 if $\mathcal{P}$ has type $(3^+,2^+,*)$, then $w_G(\mathcal{P})\geq 25$;  
 
	\item \label{lem:6_partition_type_value_3311} 
 if $\mathcal{P}$ has type $(3^+,3^+,*)$, then $w_G(\mathcal{P})\geq 32$.
\end{enumerate}
\end{lemma}

\begin{proof}
	Let $H:=G[P_1]$ and let $\mathcal{Q}$ be a partition of $V(H)$ with parts $Q_1,Q_2,\ldots,Q_s$. Since $\mathcal{Q}\cup\mathcal{P}\backslash P_1$ is a refinement of $\mathcal{P}$, we have 
    \begin{align*}
    w_G(\mathcal{Q}\cup\mathcal{P}\backslash P_1) &= \sum\limits_{k=1}^{s}d_G(Q_k)+\sum\limits_{i=2}^{t}d_G(P_i)-23(t-1+s)+42\\
	&=\Big(\sum\limits_{k=1}^{s}d_H(Q_k)+d_G(P_1)\Big)+\sum\limits_{i=2}^{t}d_G(P_i)-23(t-1+s)+42\\
	&=\Big(\sum\limits_{k=1}^{s}d_H(Q_k)-23s+42\Big)+\Big(\sum\limits_{i=1}^{t}d_G(P_i)-23t+42\Big)+23-42 \\
	&=w_H(\mathcal{Q})+w_G(\mathcal{P})-19.
    \end{align*}
    Thus,     \begin{equation}\label{equ:refinement}
        w_{H}(\mathcal{Q})=w_G(\mathcal{Q}\cup\mathcal{P}\setminus P_1)-w_G(\mathcal{P})+19.
    \end{equation}
	
    Each of (a)--(f) is proved in the same way. We assume the statement is false and use Equation~\eqref{equ:refinement} to 
    show that $H$ is an $\mathcal{N}$-good subgraph of $G$. Thus, \Cref{thm:main_statement_SZ9_k8}(1) holds, so $G$ is not a counterexample after all, a contradiction.
    We use earlier parts of the lemma to prove later parts, which is why we phrase many of its parts in such generality even
    though many of these (more general) cases are subsumed by later parts.
   
    For~\ref{lem:1_partition_type_value_2111}, suppose $\mathcal{P}$ is a partition with type $(2^+,1^+,\ast)$ such that $w_{G}(\mathcal{P})\leq 8$ and $|\mathcal{P}|$ is maximized. 
    For the trivial partition $\mathcal{Q}_0$ of $V(H)$, by~\Cref{equ:refinement}, 
    $w_{H}(\mathcal{Q}_0)=w_G(\mathcal{Q}_0\cup\mathcal{P}\backslash P_1)-w_G(\mathcal{P})+19
    \geq w(G)-w_{G}(\mathcal{P})+19\geq 0-8+19=11$. 
    Thus \Cref{ob:partition_bound_notin}\ref{ob:3_12_N} implies $H/\mathcal{Q}_0\notin\mathcal{N}$. 
    Every nontrivial partition $\mathcal{Q}$ has type $(2^+,1^+,\ast)$ and $\mathcal{Q}\cup\mathcal{P}\backslash P_1$ (as a partition of $V(G)$) also has type 
    $(2^+,1^+,\ast)$ with $|\mathcal{Q}\cup\mathcal{P}\backslash P_1|>|\mathcal{P}|$. 
    By the maximality of $|\mathcal{P}|$, we have $w_{G}(\mathcal{Q}\cup\mathcal{P}\backslash P_1)\geq 9$. 
    Thus, $w_{H}(\mathcal{Q})=w_G(\mathcal{Q}\cup\mathcal{P}\setminus P_1)-w_G(\mathcal{P})+19\geq 9-8+19=20$ so 
    \Cref{ob:partition_bound_notin}\ref{ob:1_14_WN} gives $H/\mathcal{Q}\notin\mathcal{N}\cup\mathcal{W}^*$. 
    Combining the cases above, we get $w(H)\geq 11$. Since $H/\mathcal{Q}_0\notin\mathcal{N}$ and
    $H/\mathcal{Q}\notin\mathcal{N}\cup\mathcal{W}^*$ for all other partitions $\mathcal{Q}$, 
    by definition $H$ is an $\mathcal{N}$-good subgraph of $G$.
	
	For~\ref{lem:2_partition_type_value_3111}, suppose $\mathcal{P}$ is a partition with type $(3^+,1^+,*)$ 
    such that $w_{G}(\mathcal{P})\leq 15$. We consider a partition $\mathcal{Q}$ of $V(H)$. 
    If $\mathcal{Q}$ is trivial, then \Cref{equ:refinement} gives $w_{H}(\mathcal{Q})\geq 0-15+19=4$. 
    As the trivial partition has at least $3$ parts, \Cref{ob:partition_bound_notin}\ref{ob:4_5_N} implies 
    $H/\mathcal{Q}\notin \mathcal{N}$.
    Otherwise $\mathcal{Q}$ has type $(2^+,1^+,*)$, so $\mathcal{Q}\cup\mathcal{P}\backslash P_1$ is a partition of $V(G)$ 
    with type $(2^+,1^+,*)$. Now \ref{lem:1_partition_type_value_2111} gives 
    $w_{H}(\mathcal{Q})=w_G(\mathcal{Q}\cup\mathcal{P}\setminus P_1)- w_G(\mathcal{P})+19\geq 9-15+19=13$. 
    So \Cref{ob:partition_bound_notin}\ref{ob:1_14_WN}
    gives $H/\mathcal{Q}\notin \mathcal{N}\cup\mathcal{W}^*$.
    Thus, by definition $H$ is $\mathcal{N}$-good. 
	
	For~\ref{lem:3_partition_type_value_2211}, suppose $\mathcal{P}$ is a partition with type $(2^+,2^+,\ast)$ such that 
    $w_{G}(\mathcal{P})\leq 17$ and $|\mathcal{P}|$ is maximized. For the trivial partition $\mathcal{Q}_0$ of $V(H)$, 
    the partition $\mathcal{Q}_0\cup\mathcal{P}\backslash P_1$ of $V(G)$ has type $(2^+, 1^+, \ast)$, so 
    \ref{lem:1_partition_type_value_2111} gives
    $w_{H}(\mathcal{Q}_0)=w_G(\mathcal{Q}_0\cup\mathcal{P}\backslash P_1)-w_G(\mathcal{P})+19\geq 9-17+19=11$.
    Now \Cref{ob:partition_bound_notin}\ref{ob:3_12_N} gives $H/\mathcal{Q}_0\notin\mathcal{N}$. 
    Every nontrivial partition $\mathcal{Q}$ has type $(2^+,1^+,\ast)$ and $\mathcal{Q}\cup\mathcal{P}\backslash P_1$ has 
    type$(2^+,2^+,\ast)$, as a partition of $V(G)$, with $|\mathcal{Q}\cup\mathcal{P}\backslash P_1|>|\mathcal{P}|$. 
    By the maximality of $|\mathcal{P}|$, we have $w_G(\mathcal{Q}\cup\mathcal{P}\backslash P_1)\geq 18$, so 
    $w_{H}(\mathcal{Q})=w_G(\mathcal{Q}\cup\mathcal{P}\backslash P_1)-w_G(\mathcal{P})+19\geq 18-17+19=20$, and 
    \Cref{ob:partition_bound_notin}\ref{ob:1_14_WN} gives $H/\mathcal{Q}\notin\mathcal{N}\cup\mathcal{W}^*$. 
    By definition, $H$ is $\mathcal{N}$-good.
	
    For~\ref{lem:4_partition_type_value_4111}, suppose $\mathcal{P}$ is a partition with type $(4^+,1^+,*)$ 
    such that $w_G(\mathcal{P})\leq 19$. 
    We consider a partition $\mathcal{Q}$ of $V(H)$. If $\mathcal{Q}$ is trivial, then $w_{H}(\mathcal{Q})\geq 0-19+19=0$ and $H/\mathcal{Q}\notin\mathcal{N}$ as $|\mathcal{Q}|\geq 4$. If $\mathcal{Q}$ has type $(2^+,1^+,*)$ and $|\mathcal{Q}|\geq 3$, then $\mathcal{Q}\cup\mathcal{P}\backslash P_1$ (as a partition of $V(G)$) has type $(2^+,1^+,*)$, so~\ref{lem:1_partition_type_value_2111} gives $w_{H}(\mathcal{Q})\geq 9-19+19=9$. 
    Now~\Cref{ob:partition_bound_notin}\ref{ob:2_7_WN} gives $H/\mathcal{Q}\notin\mathcal{N}\cup\mathcal{W}^*$. 
    Otherwise $\mathcal{Q}$ has type either $(3^+,1^+,*)$ or $(2^+,2^+,*)$, so \ref{lem:2_partition_type_value_3111} 
    and \ref{lem:3_partition_type_value_2211} give $w_{H}(\mathcal{Q})\geq \min\{16,18\}-19+19=16$. So~\Cref{ob:partition_bound_notin}\ref{ob:1_14_WN} gives $H/\mathcal{Q}\notin\mathcal{N}\cup\mathcal{W}^*$. 
    By definition, $H$ is $\mathcal{N}$-good.
	
	For~\ref{lem:5_partition_type_value_3211}, suppose $\mathcal{P}$ is a partition with type $(3^+,2^+,*)$ 
    such that $w_G(\mathcal{P})\leq 24$. For the trivial partition $\mathcal{Q}_0$ of $V(H)$, the partition   $\mathcal{Q}_0\cup\mathcal{P}\backslash P_1$ of $V(G)$ has type $(2^+, 1^+, \ast)$, so \ref{lem:1_partition_type_value_2111} 
    gives $w_{H}(\mathcal{Q}_0)\geq 9-24+19=4$.
    As $|\mathcal{Q}_0|\geq 3$, \Cref{ob:partition_bound_notin}\ref{ob:4_5_N} gives $H/\mathcal{Q}_0\notin \mathcal{N}$.
    Every nontrivial partition $\mathcal{Q}$ has type $(2^+,1^+,*)$ so the partition $\mathcal{Q}\cup\mathcal{P}\backslash P_1$ 
    of $V(G)$ has type $(2^+,2^+,*)$; thus \ref{lem:3_partition_type_value_2211} gives $w_{H}(\mathcal{Q})\geq 18-24+19=13$. 
    Now~\Cref{ob:partition_bound_notin}\ref{ob:1_14_WN} gives $H/\mathcal{Q}\notin\mathcal{N}\cup\mathcal{W}^*$. 
    By definition, again $H$ is $\mathcal{N}$-good. 
	
    For~\ref{lem:6_partition_type_value_3311}, assume $\mathcal{P}$ is a partition with type $(3^+,3^+,*)$ 
    such that $w_G(\mathcal{P})\leq 31$.  For the trivial partition $\mathcal{Q}_0$, the partition  
    $\mathcal{Q}_0\cup\mathcal{P}\backslash P_1$ of $V(G)$ has type $(3^+, 1^+, \ast)$, so \ref{lem:2_partition_type_value_3111} 
    gives $w_{H}(\mathcal{Q}_0)\geq 16-31+19=4$. Since $|\mathcal{Q}_0|\geq 3$, \Cref{ob:partition_bound_notin}\ref{ob:4_5_N} gives $H/\mathcal{Q}_0\notin\mathcal{N}$. For every partition $\mathcal{Q}$ of type $(2^+,1^+,*)$, the partition $\mathcal{Q}\cup\mathcal{P}\backslash P_1$ of $V(G)$ has type $(3^+,2^+,*)$, so \ref{lem:5_partition_type_value_3211} gives $w_{H}(\mathcal{Q})\geq 25-31+19=13$.
    Now \Cref{ob:partition_bound_notin}\ref{ob:1_14_WN} gives $H/\mathcal{Q}\notin\mathcal{N}\cup\mathcal{W}^*$. 
    By definition, $H$ is $\mathcal{N}$-good.
\end{proof}

Since $v(G)\geq 5$, \Cref{lem:partition_type_value} gives the following lower bound on the edge-connectivity of $G$.

\begin{lemma}\label{lem:edge_connectivity}
	The graph $G$ is $12$-edge-connected. If $[X,X^c]$ is an edge-cut of $G$, then  
	\begin{enumerate}[label=(\arabic*)]
		\item\label{lem:edge_cut_14} when $|X|\geq3$ and $|X^c|\geq2$, we have $|[X,X^c]|\geq15$;
		\item\label{lem:edge_cut_18} when $|X|\geq3$ and $|X^c|\geq3$, we have $|[X,X^c]|\geq18$.
	\end{enumerate}
\end{lemma}

\begin{proof}
	Note that $\mathcal{P}=\{X,X^c\}$ is a partition of $V(G)$. Since $v(G)\geq 5$, we know $\mathcal{P}$ has type either $(4^+,1^+)$ or $(3^+,2^+)$. By~\Cref{lem:partition_type_value}\ref{lem:4_partition_type_value_4111} and \ref{lem:5_partition_type_value_3211},
    we have $w_{G}(\mathcal{P})=2|[X,X^c]|-23\times 2+42\geq \min\{20,25\},$
	which implies $|[X,X^c]|\geq 12$; thus $G$ is $12$-edge-connected. 
	
	For~\ref{lem:edge_cut_14}, by~\Cref{lem:partition_type_value}\ref{lem:5_partition_type_value_3211}, we have $2|[X,X^c]|-23\times 2+42\geq 25$. Thus $|[X,X^c]|\geq \lceil 29/2\rceil = 15$. For~\ref{lem:edge_cut_18}, by~\Cref{lem:partition_type_value}\ref{lem:6_partition_type_value_3311}, we have $2|[X,X^c]|-23\times 2+42\geq 32$. Thus $|[X,X^c]|\geq18$.
\end{proof}

Next, we give the first configuration forbidden in the minimal counterexample $G$. 

\begin{lemma}\label{cla:T117}
	The graph $G$ contains no $T_{1,1,7}$.
\end{lemma}

\begin{proof}
	Suppose to the contrary that $G$ contains $T_{1,1,7}$ as a subgraph as shown in~\Cref{fig:T117}. We lift an edge-pair ($xv, vy$) 
    at $v$, contract the resulting $8K_2$, denote by $w_{xy}$ the vertex formed from the contraction, and denote the resulting graph 
    by $G'$. Note that $v(G')=v(G)-1\geq 4$ and $e(G')=e(G)-9$. As $8K_2$ is an $\mathcal{N}$-good graph, to 
    contradict \Cref{thm:main_statement_SZ9_k8}\ref{thm:2_main_statement_SZ9_k8}, it suffices to prove that $G'$ is $\mathcal{S}$-good. 
    
    For the trivial partition $\mathcal{P}'_0$ of $V(G')$ and the trivial partition $\mathcal{P}_0$ of $V(G)$, we have $w_{G'}(\mathcal{P}'_0)= w_{G}(\mathcal{P}_0)-2\times 9+ 23\times 1\geq w(G)-18+23=5$. Moreover, $|\mathcal{P}'_0|=v(G')\geq 4$, which gives $G'/\mathcal{P}'_0\notin \mathcal{N}\cup\mathcal{W}^*$. 
    Let $\mathcal{P}'$ be a nontrivial partition of $V(G')$ and $\mathcal{P}$ be its corresponding partition of $V(G)$.
    For every partition $\mathcal{P}'$ with type $(2^+,1^+,\ast)$, the corresponding partition $\mathcal{P}$
    has type either $(3^+,1^+,\ast)$ or $(2^+,2^+,\ast)$; thus 
    \Cref{lem:partition_type_value}\ref{lem:2_partition_type_value_3111} and \ref{lem:3_partition_type_value_2211} 
    give $w_G(\mathcal{P})\geq \min\{16, 18\}=16$.
    Now $w_{G'}(\mathcal{P}')= w_{G}(\mathcal{P})-2\times 2\geq 12$. 
    Furthermore, $G'$ is $10$-edge-connected by~\Cref{lem:edge_connectivity}. 
    So \Cref{ob:partition_bound_notin}\ref{ob:5_10-edge-connected} gives $G'/\mathcal{P}'\notin\mathcal{N}\cup\mathcal{W}^*$. 
    Noting that $w(G')\geq 5$, we conclude that $G'$ is $\mathcal{S}$-good, as desired.
\end{proof}

The next lemma helps us handle vertices in $G$ of small degree. 

\begin{lemma}\label{lem:14_lift} 
    Let $G$ contain a vertex $v$ with $d_{G}(v)\leq 14$. Let $G'$ be a graph formed from $G$ by lifting $\alpha$ 
    edge-pairs at $v$ such that $d_{G}(v)-\alpha\leq 11$ and $d_{G}(v)-2\alpha\geq 8$, and let $G''=G'-v$. If $G''$ contains 
    no $9K_2$, then $G''$ is $\mathcal{S}$-good and $G'/\mathcal{P}\notin\mathcal{N}$ for every partition $\mathcal{P}$ of $V(G')$.
\end{lemma}

\begin{proof}
    Note that $v(G'')=v(G)-1\geq4$, and $e(G'')=e(G)-(d_{G}(v)-\alpha)$. 
    We begin with the following claim, which implies that $G''$ is $\mathcal{S}$-good.

    \smallskip
    \noindent 
    \emph{Claim: $w(G'')\ge 0$ and $G''/\mathcal{P}''\notin\mathcal{N}\cup\mathcal{W}^*$ for every 
    partition $\mathcal{P}''$ of $V(G'')$.} 

    \smallskip
    The trivial partition $\mathcal{P}''_0$ of $V(G'')$ satisfies 
    $w_{G''}(\mathcal{P}_0'')= w_{G}(\mathcal{P}_0)-2\times (d_{G}(v)-\alpha)+23\times 1\geq w(G)-2\times 11+23\geq 1$, where $\mathcal{P}_0$ is the trivial partition of $V(G)$.
    Further, $|\mathcal{P}''_0|=v(G'')\ge 4$, so $G''/\mathcal{P}''\notin\mathcal{N}\cup\mathcal{W}^*$. Let $\mathcal{P}''=\{P_1, P_2, \ldots P_t\}$ be a nontrivial partition of $V(G'')$.
    If $\mathcal{P}''$ is a partition of $V(G'')$ with type $(4^+,1^+,*)$ or type $(2^+,2^+,*)$, then $\mathcal{P}''\cup\{v\}$ also has type  $(4^+,1^+,*)$ or $(2^+,2^+,*)$ as a partition of $V(G)$. Then \Cref{lem:partition_type_value}\ref{lem:3_partition_type_value_2211} and \ref{lem:4_partition_type_value_4111} imply $w_{G''}(\mathcal{P}'')
    \geq w_{G}(\mathcal{P}''\cup\{v\})-2\times d_{G}(v)+23\times 1 \geq \min\{18,20\}-28+23=13$.    So~\Cref{ob:partition_bound_notin}\ref{ob:1_14_WN} gives $G''/\mathcal{P}''\notin\mathcal{N}\cup\mathcal{W}^*$. Thus we assume $|P_1|\in \{2,3\}$ and $|P_i|=1$ for all $i\ge 2$.
    First consider such a partition $\mathcal{P}''$ with type $(2,1,*)$ and let $P_1=\{x, y\}$. Since $G''$ contains no $9K_2$, we have
    $w_{G''}(\mathcal{P}'')\geq w_{G''}(\mathcal{P}''_0)-2\times \mu_{G''}(xy)+23\times 1 \geq1-16+23=8$. 
    Since $v(G'')\ge 4$, we have $|\mathcal{P}''|\ge 3$,
    so \Cref{ob:partition_bound_notin}\ref{ob:2_7_WN} implies $G''/\mathcal{P}''\notin\mathcal{N}\cup\mathcal{W}^*$.
    
    Instead we assume $\mathcal{P}''$ has type $(3,1,*)$. For every type $(3,1,*)$ partition $\mathcal{P}''$ of $V(G'')$, we know that 
    $\mathcal{P}''\cup \{v\}$ is a type $(3,1,*)$ partition of $V(G)$, so 
    \Cref{lem:partition_type_value}\ref{lem:2_partition_type_value_3111} implies
    $w_{G''}(\mathcal{P}'')\geq w_{G}(\mathcal{P}''\cup\{v\}) - 2\times d_{G}(v) + 23\times 1 \geq 16-28+23 = 11$.
    Thus \Cref{ob:partition_bound_notin}\ref{ob:3_12_N} gives $G''/\mathcal{P}'' \notin\mathcal{N}\cup\mathcal{W}^*\backslash 
    \{8K_2\}$.  So we assume $G''/\mathcal{P}''=8K_2$ where $\mathcal{P}''=\{P_1, P_2\}$ with $|P_1|=3$ and $|P_2|=1$. 
    Moreover, this case holds only if we lifted $\alpha$ edge-pairs at $v$ such that $\alpha$ edges have been added into 
    $G[P_1]$ and $d_{G}(v)=14$; otherwise, the inequality above can be improved to $w_{G''}(\mathcal{P}'')\geq 13$, and we are done.  
    Thus $\alpha=3$ and $v$ has at least two neighbours in $P_1$. 

Note that $v(G)=5$ and $0\le w(G)\le \sum_{v\in V(G)}d(v)-23v(G)+42 = 2e(G)-115+42$, so
    $e(G)\geq \lceil 73/2\rceil = 37$, and then $e(G[P_1])=e(G)-d_{G}(v)-e_{G}(P_1,P_2)\geq 37-14-8=15$. 
    That is, $G[P_1]$ is a multitriangle $T_{a,b,c}$ with $a+b+c\geq 15$. 
    If $a+b+c\geq 16$, then $\delta(G[P_1])\geq 9$ because $G$ has no $8K_2$. 
    Thus $G[P_1]$ is an $\mathcal{N}$-good subgraph of $G$, which contradicts~\Cref{thm:main_statement_SZ9_k8}\ref{thm:1_main_statement_SZ9_k8}.
    So assume $a+b+c=15$. Note that $e_{G}(P_1,P_2\cup\{v\})\geq 15$ and $d_{G}(P_2)\geq 12$ by~\Cref{lem:edge_connectivity}. Since $G$ contains no $T_{1,1,7}$ (by~\Cref{cla:T117}), we know $\mu(G[P_1])\le 6$, 
    so $\delta(G[P_1])\geq 9$. 
    In this case, we lift an edge-pair at $v$ to add an edge into $G[P_1]$ and denote it by $G_1$. It is easy to see that $G_1$ is a multitriangle (induced by $P_1$) with $e(G_1)\geq 16$ and $\delta(G_1)\geq 9$, and thus $G_1$ is an $\mathcal{N}$-good graph. Then we contract $G_1$, and note that the graph formed by contraction is an $\mathcal{S}$-good graph on three vertices with at
    least $20$ edges and minimum degree at least $9$. Hence, it contradicts~\Cref{thm:main_statement_SZ9_k8}\ref{thm:2_main_statement_SZ9_k8}. 
    Therefore, $w_{G''}(\mathcal{P}'')\ge 0$ and $G''/\mathcal{P}''\notin\mathcal{N}\cup\mathcal{W}^*$ 
    for every partition $\mathcal{P}''$, which proves our claim.
    
    \medskip
    Next we prove that $G'/\mathcal{P}'\notin\mathcal{N}$ for every partition $\mathcal{P}'$ of $V(G')$. We consider two cases 
    based on the size of $\mathcal{P}'$. First suppose that $\mathcal{P}'$ is a partition of $V(G')$ with two parts $P_1$ and 
    $P_2$.  By symmetry, we assume $v\in P_1$.  If $P_1=\{v\}$, then $e(P_1,P_2) = d_{G'}(v) \ge 8$, so 
    $G'/\mathcal{P}'\notin \mathcal{N}$.  Thus, we assume $|P_1|\ge 2$.  Hence, $\{P_1-v,P_2\}$ is a partition of $V(G'')$. 
    Since $G''/\{P_1-v,P_2\}\notin\mathcal{N}\cup\mathcal{W}^*$, we know that $e_{G'}(P_1-v,P_2)=e_{G''}(P_1-v,P_2)\geq 9$, and 
    thus $e_{G'}(P_1,P_2)\geq 9$. Hence, $G'/\mathcal{P}'\notin\{\alpha K_2:\alpha\leq 7\}$. 
    Now instead we consider a partition $\mathcal{P}'$ of $V(G')$ with three parts $P_1$, $P_2$, and $P_3$. If $P_1=\{v\}$, then 
    $\{P_2,P_3\}$ is a partition of $V(G'')$. Since $e_{G'}(P_2,P_3)=e_{G''}(P_2,P_3)\geq 9$ and $d_{G'}(v)\geq 8$, 
    we know $e(G'/\mathcal{P}')\ge 9+8=17$, so $G'/\mathcal{P}'\notin\{T_{a,b,c}:a+b+c\leq 15\}$. 
    If $v\in P_1$ and $|P_1|\geq 2$, then we consider a partition of $V(G'')$ with parts $P_1-v, P_2, P_3$; 
    call it $\mathcal{P}''$. By the above claim, $G''/\mathcal{P}''\notin\{T_{a,b,c}:a+b+c\leq 15\}$; 
    thus $G'/\mathcal{P}'\notin\{T_{a,b,c}:a+b+c\leq 15\}$. 
    Therefore, there is no partition $\mathcal{P}'$ of $V(G')$ such that $G'/\mathcal{P}'\in\mathcal{N}$. 
\end{proof}

The next result is the first application of the above lemma. In fact, we show that $\delta(G)\ge 14$.

\begin{lemma}\label{lem:edge_connectivity_14}
	The graph $G$ is $14$-edge-connected.
\end{lemma}

\begin{proof}
    It suffices to prove that $\delta(G)\geq 14$. The lemma then follows from 
    the fact that $v(G)\geq 5$, as if $[X,X^c]$ is an edge-cut with $|X|\geq 2$ and $|X^c|\geq 3$, then \Cref{lem:edge_connectivity}\ref{lem:edge_cut_14} gives $|[X,X^c]|\geq 15$.
 
    Suppose to the contrary that $\delta(G)\leq 13$.  By \Cref{lem:edge_connectivity}, we have $\delta(G)\in\{12,13\}$. Let $v$ be a vertex of $G$ with $d_G(v)=\delta(G)$. We lift two edge-pairs at $v$ to form a new graph $G'$, delete the vertex $v$, and denote the resulting graph by $G''$. Note that $d_G(v)-2\in\{10,11\}$ and $d_G(v)-2\times2 \in \{8,9\}$. Since $G$ contains no $8K_2$ and $G$
contains no $T_{1,1,7}$ (by~\Cref{cla:T117}), we know that $\mu_{G''}(xy)\leq 8$ for every two vertices $x,y\in V(G'')$. Thus, by~\Cref{lem:14_lift}, $G''$ is $\mathcal{S}$-good and $G'/\mathcal{P}\notin\mathcal{N}$ for every partition $\mathcal{P}$ of $V(G')$, contradicting~\Cref{thm:main_statement_SZ9_k8}\ref{thm:3_main_statement_SZ9_k8}.
\end{proof}

After improving our bounds on the edge-connectivity of $G$, we can lift more edge-pairs in $G$. 
We make this more precise in the next lemma.

\begin{lemma}
    \label{cla:lift_edge_pairs}
	Let $G_1$ be a graph formed from $G$ by lifting $X$ edge-pairs, $Y$ edge-triples, and $Z$ edge-quadruples
    with $X+Y+Z\leq2$.  If also $X+2Y+3Z\leq 3$, then $G_1$ contains no $8K_2$.
\end{lemma}

\begin{proof}
    Recall that $G$ has no $T_{1,1,7}$ by \Cref{cla:T117}; since at most $X+Y+Z\leq 2$ new edges may be created by lifting, $G_1$ has no $9K_2$.
	Suppose the lemma is false; that is, $G_1$ contains $8K_2$. We contract $8K_2$ and denote the resulting graph by $G'$. 
    Since $X+2Y+3Z\leq 3$, after lifting and contracting, $v(G')=v(G)-1\geq4$ and $e(G')\geq e(G)-(8+X+2Y+3Z)\geq e(G)-11$. By~\Cref{lem:edge_connectivity_14}, $G_1$ has edge-connectivity at least $14-(2X+3Y+4Z)\geq 14-(2+3)=9$ and thus $G'$ is also $9$-edge-connected. 

    For the trivial partition $\mathcal{P}_0'$ of $V(G')$, we have $w_{G'}(\mathcal{P}_0')\geq w(G)-2\times 11+23\times 1 \geq 1$. 
    Because $|\mathcal{P}_0'|\geq 4$, clearly $G'/\mathcal{P}_0'\notin\mathcal{N}\cup \mathcal{W}^*$. 
    Let $\mathcal{P}'$ be a nontrivial partition of $V(G')$ and $\mathcal{P}$ be the corresponding partition of $V(G)$. 
    If $\mathcal{P}'$ has type $(2^+,1^+,\ast)$, then $\mathcal{P}$ is a partition of $V(G)$ with type 
    either $(3^+,1^+,\ast)$ or $(2^+,2^+,\ast)$. 
    Hence, 
    \Cref{lem:partition_type_value}\ref{lem:2_partition_type_value_3111} and \ref{lem:3_partition_type_value_2211} 
    give $w_G(\mathcal{P})\geq \min\{16, 18\}=16$.
    Let $m$ denote the number of edges that are counted in $\omega_G(\mathcal{P})$ but not in $\omega_{G'}(\mathcal{P}')$. 
    Note that $m\leq 2X+3Y+4Z\leq 5$. 
    Thus, $w_{G'}(\mathcal{P}')\geq w_G(\mathcal{P})-2\times m\geq 16-2m\geq 6$. 
    Because $G'$ is $9$-edge-connected,  \Cref{ob:partition_bound_notin}\ref{ob:5_10-edge-connected} gives 
    $G'/\mathcal{P}'\notin\mathcal{N}\cup\mathcal{W}^*$. 
    Since $w(G')\geq 1$, $G'$ is an $\mathcal{S}$-good graph, 
    contradicting~\Cref{thm:main_statement_SZ9_k8}\ref{thm:2_main_statement_SZ9_k8}.
\end{proof}

\begin{figure}[!htbp]
	\centering
	\begin{subfigure}[t]{.34\textwidth}
		\centering
	    \begin{tikzpicture}[scale=.55]
		    \draw [line width=0.8pt, black] (0,3) to (2,0);
		    \draw [line width=0.8pt, black] (0,3) to (-2,0);		  
  		    \draw [line width=0.8pt, black] (2,0) to (-2,0);
  		    \draw [bend left=10,line width=0.8pt, black] (2,0) to (-2,0);
		    \draw [bend right=10,line width=0.8pt, black] (2,0) to (-2,0);
		    \draw [bend left=20,line width=0.8pt, black] (2,0) to (-2,0);
		    \draw [bend right=20,line width=0.8pt, black] (2,0) to (-2,0);
		    \draw [bend left=30,line width=0.8pt, black] (2,0) to (-2,0);
		    \draw [bend right=30,line width=0.8pt, black] (2,0) to (-2,0);
  
	    	\draw [fill=black,line width=0.2pt] (-2,0) node[left=0.5mm] {$x$} circle (4pt);
	    	\draw [fill=black,line width=0.2pt] (2,0) node[right=0.5mm] {$y$} circle (4pt);
	    	\draw [fill=white,line width=0.8pt] (0,3) node[above] {$v$} circle (3pt);  
	    \end{tikzpicture}
	    \caption{$T_{1,1,7}$}
	    \label{fig:T117}
	\end{subfigure}
	\begin{subfigure}[t]{.34\textwidth}
		\centering
		\begin{tikzpicture}[scale=0.55]
			\draw [line width=0.8pt] (-2,2)-- (-2,-2);%ux
			\draw [line width=0.8pt] (2,2)-- (2,-2);%uy
			\draw [line width=0.8pt] (-2,2)-- (2,2);%uv

     		\draw [line width=0.8pt, black] (2,-2) to (-2,-2);
  		    \draw [bend left=10,line width=0.8pt, black] (2,-2) to (-2,-2);
		    \draw [bend right=10,line width=0.8pt, black] (2,-2) to (-2,-2);
		    \draw [bend left=20,line width=0.8pt, black] (2,-2) to (-2,-2);
		    \draw [bend right=20,line width=0.8pt, black] (2,-2) to (-2,-2);
		    \draw [bend left=30,line width=0.8pt, black] (2,-2) to (-2,-2);
		    \draw [bend right=30,line width=0.8pt, black] (2,-2) to (-2,-2);
			
			\draw [fill=white,line width=0.8pt] (-2,2) node[left=0.5mm] {$u$} circle (3pt);
			\draw [fill=white,line width=0.8pt] (2,2) node[right=0.5mm] {$v$} circle (3pt);
			\draw [fill=black,line width=0.2pt] (-2,-2)node[left=0.5mm] {$x$} circle (4pt);
			\draw [fill=black,line width=0.2pt] (2,-2)node[right=0.5mm] {$y$} circle (4pt);	
		\end{tikzpicture}
		\caption{\footnotesize$Q_{1,1,1,7}$}
		\label{fig:Q1117}
	\end{subfigure}
	\begin{subfigure}[t]{.3\textwidth}
		\centering
		\begin{tikzpicture}[scale=0.55]
			\draw [line width=0.8pt] (-2,0.5)-- (-2,-2);%ux
			\draw [line width=0.8pt] (2,0.5)-- (2,-2);%vy
			\draw [line width=0.8pt] (-2,0.5)-- (0,2);%uw
			\draw [line width=0.8pt] (0,2)-- (2,0.5);%wv
     		\draw [line width=0.8pt, black] (2,-2) to (-2,-2);%xy
  		    \draw [bend left=10,line width=0.8pt, black] (2,-2) to (-2,-2);
		    \draw [bend right=10,line width=0.8pt, black] (2,-2) to (-2,-2);
		    \draw [bend left=20,line width=0.8pt, black] (2,-2) to (-2,-2);
		    \draw [bend right=20,line width=0.8pt, black] (2,-2) to (-2,-2);
		    \draw [bend left=30,line width=0.8pt, black] (2,-2) to (-2,-2);
		    \draw [bend right=30,line width=0.8pt, black] (2,-2) to (-2,-2);

			\draw [fill=white,line width=0.8pt] (-2,0.5) node[left=0.5mm] { $u$} circle (3pt);
			\draw [fill=white,line width=0.8pt] (2,0.5) node[right=0.5mm] {$v$} circle (3pt);
			\draw [fill=white,line width=0.8pt] (0,2) node[above=0.5mm] {$w$} circle (3pt);
			\draw [fill=black,line width=0.8pt] (-2,-2)node[left=0.5mm] {$x$} circle (4pt);
			\draw [fill=black,line width=0.8pt] (2,-2)node[right=0.5mm] {$y$} circle (4pt);
		\end{tikzpicture}
		\caption{\footnotesize$V_{1,1,1,1,7}$}
		\label{fig:V11117}
	\end{subfigure}
	\begin{subfigure}[t]{.34\textwidth}
		\centering
		\begin{tikzpicture}[scale=0.55]
			\draw [bend right=10, line width=0.8pt] (0,3)to (-2,0);%ux
			\draw [bend left=10, line width=0.8pt] (0,3)to (-2,0);%ux
			
			\draw [bend right=10, line width=0.8pt] (0,3)to (2,0);%uy
			\draw [bend left=10, line width=0.8pt] (0,3)to (2,0);%uy
		    \draw [bend left=5,line width=0.8pt, black] (2,0) to (-2,0);%xy
		    \draw [bend right=5,line width=0.8pt, black] (2,0) to (-2,0);
		    \draw [bend left=15,line width=0.8pt, black] (2,0) to (-2,0);
		    \draw [bend right=15,line width=0.8pt, black] (2,0) to (-2,0);
		    \draw [bend left=25,line width=0.8pt, black] (2,0) to (-2,0);
		    \draw [bend right=25,line width=0.8pt, black] (2,0) to (-2,0);			

			\draw [fill=white, line width=0.8pt] (0,3) node[above] {$u$} circle (3pt);
			\draw [fill=black, line width=0.2pt] (-2,0)node[left=0.5mm] {$x$} circle (4pt);
			\draw [fill=black, line width=0.2pt] (2,0)node[right=0.5mm] {$y$} circle (4pt);
		\end{tikzpicture}
		\caption{\footnotesize$T_{2,2,6}$}
		\label{fig:T226}
	\end{subfigure}
	\begin{subfigure}[t]{.34\textwidth}
		\centering
		\begin{tikzpicture}[scale=0.65]
			\draw [line width=0.8pt] (0.,0.)-- (-1.5,-2);%ux
			\draw [line width=0.8pt] (0.,0.)-- (1.5,-2);%uy
			\draw [line width=0.8pt] (0.,-4)-- (-1.5,-2);%vx
			\draw [line width=0.8pt] (0.,-4)-- (1.5,-2);%vy
		    \draw [bend left=6,line width=0.8pt, black] (1.5,-2) to (-1.5,-2);%xy
		    \draw [bend right=6,line width=0.8pt, black] (1.5,-2) to (-1.5,-2);
		    \draw [bend left=18,line width=0.8pt, black] (1.5,-2) to (-1.5,-2);
		    \draw [bend right=18,line width=0.8pt, black] (1.5,-2) to (-1.5,-2);
		    \draw [bend left=30,line width=0.8pt, black] (1.5,-2) to (-1.5,-2);
		    \draw [bend right=30,line width=0.8pt, black] (1.5,-2) to (-1.5,-2);	   

			\draw [fill=white,line width=0.8pt] (0.,0.) node[above] { $u$} circle (3pt);
			\draw [fill=white,line width=0.8pt] (0.,-4) node[below] { $v$} circle (3pt);
			\draw [fill=black,line width=0.2pt] (-1.5,-2)node[left=0.5mm] { $x$} circle (4pt);
			\draw [fill=black,line width=0.2pt] (1.5,-2)node[right=0.5mm] { $y$} circle (4pt);
		\end{tikzpicture}
		\caption{\footnotesize$T^o_{1,1,7}$}
		\label{fig:T117o}
	\end{subfigure}
	\begin{subfigure}[t]{.3\textwidth}
		\centering
		\begin{tikzpicture}[scale=0.65]
			\draw [line width=0.8pt] (-1.5,0)-- (-1.5,-2);%ux
			\draw [line width=0.8pt] (1.5,0)-- (1.5,-2);%uy
			\draw [line width=0.8pt] (-1.5,0)-- (1.5,0);%uv
			
			\draw [line width=0.8pt] (0.,-4)-- (-1.5,-2);%wx
			\draw [line width=0.8pt] (0.,-4)-- (1.5,-2);%wy
		    \draw [bend left=6,line width=0.8pt, black] (1.5,-2) to (-1.5,-2);%xy
		    \draw [bend right=6,line width=0.8pt, black] (1.5,-2) to (-1.5,-2);
		    \draw [bend left=18,line width=0.8pt, black] (1.5,-2) to (-1.5,-2);
		    \draw [bend right=18,line width=0.8pt, black] (1.5,-2) to (-1.5,-2);
		    \draw [bend left=30,line width=0.8pt, black] (1.5,-2) to (-1.5,-2);
		    \draw [bend right=30,line width=0.8pt, black] (1.5,-2) to (-1.5,-2);

			\draw [fill=white,line width=0.8pt] (-1.5,0) node[left=0.5mm] {$u$} circle (3pt);
			\draw [fill=white,,line width=0.8pt] (1.5,0) node[right=0.5mm] {$v$} circle (3pt);
			\draw [fill=black,line width=0.2pt] (-1.5,-2)node[left=0.5mm] {$x$} circle (4pt);
			\draw [fill=black,line width=0.2pt] (1.5,-2)node[right=0.5mm] {$y$} circle (4pt);
			\draw [fill=white,line width=0.8pt] (0.,-4) node[below] {$w$} circle (3pt);
		\end{tikzpicture}
		\caption{\footnotesize$Q^o_{1,1,1,7}$}
		\label{fig:Q1117o}
	\end{subfigure}
	\caption{\small Some forbidden configurations in $G$: (a) for~\Cref{cla:T117}, and (b)-(f) for~\Cref{cla:configurations_pairs_at most5}}
	\label{fig:configurations_pairs_at most5}
\end{figure}
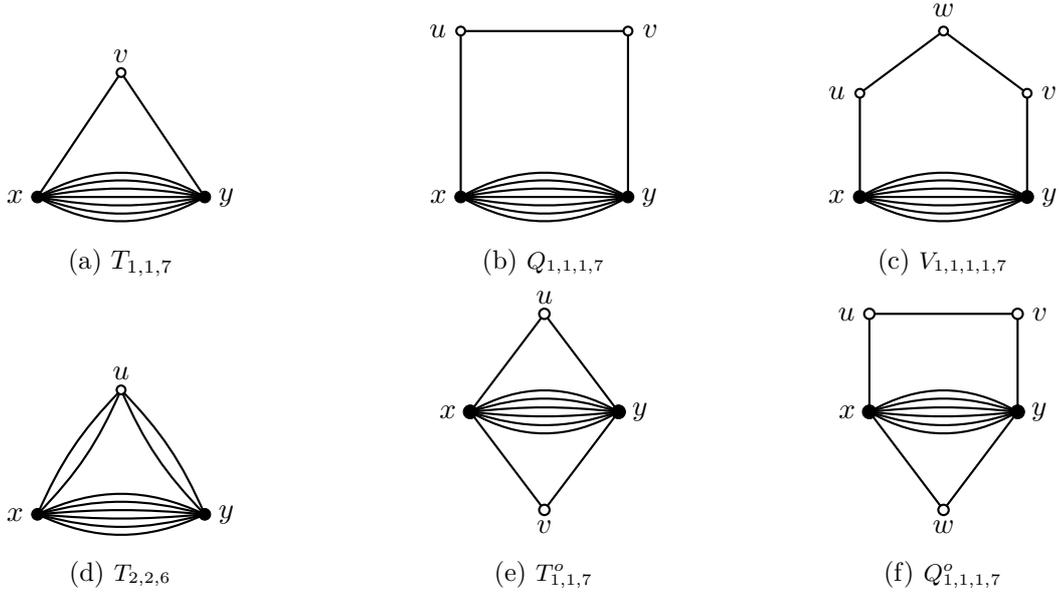

\begin{lemma}\label{cla:configurations_pairs_at most5}
	$G$ contains none of $Q_{1,1,1,7}$, $V_{1,1,1,1,7}$, $T_{2,2,6}$, $T_{1,1,7}^o$, and $Q_{1,1,1,7}^o$. 
\end{lemma}
\begin{proof}
    If $G$ contains one of these 5 configurations, shown in~\Cref{fig:Q1117}--\ref{fig:Q1117o}, then we lift some edge-pair, edge-triple, edge-quadruple, or combination of these as allowed in \Cref{cla:lift_edge_pairs}.
    We lift at most $5$ edges, so the resulting graph $G_1$ satisfies the hypotheses of \Cref{cla:lift_edge_pairs}; namely, $X+Y+Z\leq 2$ and $X+2Y+3Z\leq 3$. 
    However, we have 
    $\mu_{G_1}(xy)=8$. This contradicts~\Cref{cla:lift_edge_pairs}, which implies the result.
\end{proof}

\begin{lemma}\label{cla:T226o}
	$G$ contains no $T^{o}_{2,2,6}$.
\end{lemma}

\begin{proof}
	Suppose $G$ has a copy of $T^{o}_{2,2,6}$ as shown in~\Cref{fig:T226o}. We lift an edge-pair $(xv,vy)$ at $v$ and two edge-pairs $(xu,uy)$ at $u$, contract the resulting $8K_2$ into a new vertex $w_{xy}$, and denote the resulting graph by $G'$.
 
    First we show that $G'$ is $9$-edge-connected. By~\Cref{lem:edge_connectivity_14}, $d_{G'}(u)=d_{G}(u)-4\geq 14-4=10$ and $d_{G'}(v)=d_{G}(v)-2\geq 12$. Since $G$ contains no $T_{2,2,6}$ by~\Cref{cla:configurations_pairs_at most5}, $\mu_{G}(xy)=5$. 
    Thus, $d_{G'}(w_{xy})=d_{G}(x)+d_{G}(y)-2\times\mu_{G}(xy)-6\geq 14+14-2\times5-6=12$. 
    Moreover, $d_{G'}(z)=d_{G}(z)\geq 14$ for every vertex $z\in V(G')\setminus \{u,v,w_{xy}\}$. 
    If $[X,X^c]$ is an edge-cut of $G'$ with $|X|\ge2$ and $|X^c|\ge2$, then \cref{lem:edge_connectivity}\ref{lem:edge_cut_14} gives 
    $|[X,X^c]|\geq 15-6=9$. Therefore, $G'$ is $9$-edge-connected.
    
    Now we show that $G'$ is $\mathcal{S}$-good. Note that $v(G')=v(G)-1$ and $e(G')=e(G)-11$. For the trivial partition $\mathcal{P}'_0$ of $V(G')$, $w_{G'}(\mathcal{P}'_0)\geq w(G)-2\times 11+23\times 1 =1$. Since $|\mathcal{P}'_0|=v(G')\geq 4$, $G'/\mathcal{P}_0'\notin\mathcal{N}\cup \mathcal{W}^*$. Next, we consider the values of $w_{G'}(\mathcal{P}')$ for every nontrivial partition $\mathcal{P}'$. Let $\mathcal{P}$ be the partition of $V(G)$ corresponding to $\mathcal{P}'$. We consider the
    following two cases. 

    {\bf Case 1: at least one of $\bm{u}$ and $\bm{v}$ is in the same part as $\bm{w_{xy}}$.} 
    Note that $\mathcal{P}$ is type $(3^+,1^+,\ast)$. So~\Cref{lem:partition_type_value}\ref{lem:2_partition_type_value_3111} gives
    $w_{G'}(\mathcal{P}')\geq w_G(\mathcal{P})-2\times 4\geq 16-8=8$.
    Since $G'$ is $9$-edge-connected, \Cref{ob:partition_bound_notin}\ref{ob:5_10-edge-connected} implies $G'/\mathcal{P}'\notin\mathcal{N}\cup\mathcal{W}^*$.

    {\bf Case 2:
    $\bm{u}$, $\bm{v}$, and $\bm{w_{xy}}$ are in three different parts of $\bm{\mathcal{P}'}$.} 
    Now none of the 6 edges in the multiset $\{xu,xu,uy,uy,xv,vy\}$ are counted in $\omega_{G'}(\mathcal{P}')$. 
    If $\mathcal{P}'$ has type $(3^+,1^+,\ast)$ or type $(2^+,2^+,\ast)$, then $\mathcal{P}$ has either type $(4^+,1^+,\ast)$ 
    or type $(3^+,2^+,\ast)$.  
    So~\Cref{lem:partition_type_value}\ref{lem:4_partition_type_value_4111} and \ref{lem:5_partition_type_value_3211}
    imply $w_{G'}(\mathcal{P}')\geq w_G(\mathcal{P})-2\times 6\geq \min\{20,25\}-12=8$.
    Since $G'$ is $9$-edge-connected, \Cref{ob:partition_bound_notin}\ref{ob:5_10-edge-connected} implies $G'/\mathcal{P}'\notin\mathcal{N}\cup\mathcal{W}^*$. 
    Assume instead that $\mathcal{P}'$ is type $(2,1,\ast)$, 
    so $|\mathcal{P}'|\geq 3$. Now $\mathcal{P}$ is either type $(3,1,\ast)$ or type $(2,2,\ast)$. Furthermore, 
    \Cref{lem:partition_type_value}\ref{lem:2_partition_type_value_3111} and \ref{lem:3_partition_type_value_2211} give
    $w_{G'}(\mathcal{P}')\geq w_G(\mathcal{P})-2\times 6\geq \min\{16,18\}-12=4$.
    Thus,~\Cref{ob:partition_bound_notin}\ref{ob:4_5_N} implies $G'/\mathcal{P}'\notin\mathcal{N}\cup\mathcal{W}^*\backslash \{T_{a,b,c}:a+b+c=16\}$.
    Note that $G'/\mathcal{P}'\notin\mathcal{N}\cup\mathcal{W}^*$ unless $\mathcal{P}'$ is type $(2,1,\ast)$ with 
    $|\mathcal{P}'|=3$ and $G'/\mathcal{P}'\in\{T_{a,b,c}:a+b+c=16\}$.
    We now handle this exceptional case.  
    
    Since $\mathcal{P}'$ has type $(2,1,\ast)$ with $|\mathcal{P}'|=3$, we have $v(G')=4$ and $v(G)=5$. 
    Let $z$ be the vertex in $V(G')$ that is distinct from $v$, $u$, and $w_{xy}$.  
    Recall that $w(G)\geq 0$, so $e(G)\geq \lceil 73/2\rceil = 37$. By~\Cref{cla:configurations_pairs_at most5}, $G$ contains no $T_{2,2,6}$, so $\mu_{G}(xy)=5$. 
    We consider different kinds of partitions of $V(G)$ based on which part contains the vertex $z$. If $\mathcal{P}'=\{\{w_{xy}\},\{u,z\},\{v\}\}$, then $\mathcal{P}=\{\{x,y\},\{u,z\},\{v\}\}$. 
    This gives $\mu_{G}(uz)=e(G)-e(T^o_{2,2,6})-e(G'/\mathcal{P}')\geq 37-11-16 = 10$. 
    Hence, $G$ contains an $\mathcal{N}$-good subgraph $\alpha K_2$ with $\alpha\geq 10$, 
    contradicting~\Cref{thm:main_statement_SZ9_k8}\ref{thm:1_main_statement_SZ9_k8}. 
    Similarly, if $\mathcal{P}'=\{\{w_{xy}\},\{u\},\{v,z\}\}$, then $\mu_{G}(vz)\geq 10$, a contradiction. 
    So we assume instead that $z$ is in the same part as $w_{xy}$.
    That is, $\mathcal{P}' = \{\{w_{xy},z\},\{u\},\{v\}\}$, so $\mathcal{P}=\{\{x,y,z\},\{u\},\{v\}\}$. Now a short argument on the edges incident to $u$ will give a contradiction.
    In the second paragraph of this proof, we showed 
    $d_{G'}(u)\geq 10$ and $d_{G'}(v)\geq 12$.  
    Because $e(G'/\mathcal{P}')=16$, we get $e_{G'}(P_1,P_2)=16-d_{G'}(v)\leq 4$. 
    Since $G$ contains no $T^o_{1,1,7}$, we also get $e_{G'}(P_2,P_3)=\mu_{G}(uv)\leq 5$. 
    As $P_2=\{u\}$, together these give $d_{G'}(u)=e_{G'}(P_1,P_2)+e_{G'}(P_2,P_3)\leq 4+5=9$, a contradiction. 
    Therefore, $G'/\mathcal{P}'\notin \mathcal{N}\cup\mathcal{W}^*$ for every partition $\mathcal{P}'$ of $V(G')$.
    
    \smallskip
    In each case, $G'$ is $\mathcal{S}$-good, contradicting~\Cref{thm:main_statement_SZ9_k8}\ref{thm:2_main_statement_SZ9_k8}.
\end{proof}

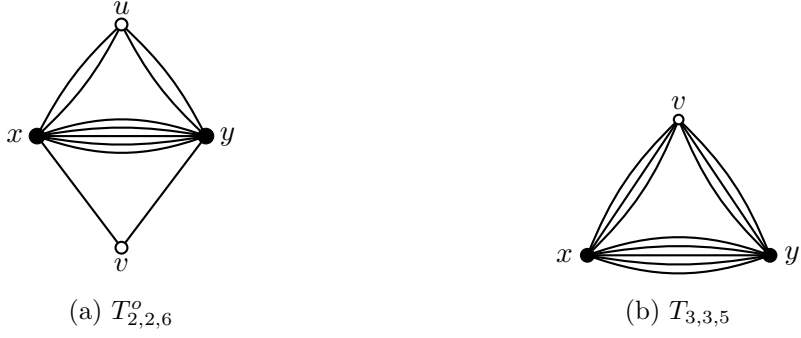
\begin{figure}[!tbp]
	\centering
	\begin{subfigure}[t]{.45\textwidth}
		\centering
		\begin{tikzpicture}[scale=0.74]
			\draw [bend right=10, line width=0.8pt] (0,0)to (-1.5,-2);%ux
			\draw [bend left=10, line width=0.8pt] (0,0)to (-1.5,-2);%ux
			
			\draw [bend right=10,line width=0.8pt] (0,0)to (1.5,-2);%uy
			\draw [bend left=10, line width=0.8pt] (0,0)to (1.5,-2);%uy
			
			\draw [line width=0.8pt] (0,-4)-- (-1.5,-2);%vx
			\draw [line width=0.8pt] (0,-4)-- (1.5,-2);%vy
			
			\draw [line width=0.8pt, black] (1.5,-2) to (-1.5,-2); %xy
			\draw [bend left=10,line width=0.8pt, black] (1.5,-2) to (-1.5,-2);
			\draw [bend right=10,line width=0.8pt, black] (1.5,-2) to (-1.5,-2);
			\draw [bend left=20,line width=0.8pt, black] (1.5,-2) to (-1.5,-2);
			\draw [bend right=20,line width=0.8pt, black] (1.5,-2) to (-1.5,-2);
            
            \draw [fill=white,line width=0.8pt] (0,0) node[above] {$u$} circle (3pt); 
            \draw [fill=white,line width=0.8pt] (0,-4) node[below] {$v$} circle (3pt);
            \draw [fill=black,line width=0.2pt] (-1.5,-2) node[left=0.5mm] {$x$} circle (4pt); 
            \draw [fill=black,line width=0.2pt] (1.5,-2) node[right=0.5mm] {$y$} circle (4pt);   
		\end{tikzpicture}
		\caption{$T^{o}_{2,2,6}$}
		\label{fig:T226o}
	\end{subfigure} 
	\begin{subfigure}[t]{.45\textwidth}
		\centering
		\begin{tikzpicture}[scale=.6]			
			\draw [line width=0.8pt, black] (0,3) to (2,0);%vy
			\draw [bend left=12,line width=0.8pt, black] (2,0) to (0,3);
			\draw [bend right=12,line width=0.8pt, black] (2,0) to (0,3);
			
			\draw [line width=0.8pt, black] (0,3) to (-2,0);%vx
			\draw [bend left=12,line width=0.8pt, black] (0,3) to (-2,0);
			\draw [bend right=12,line width=0.8pt, black] (0,3) to (-2,0);
			
			\draw [line width=0.8pt, black] (2,0) to (-2,0);%xy
			\draw [bend left=10,line width=0.8pt, black] (2,0) to (-2,0);
			\draw [bend right=10,line width=0.8pt, black] (2,0) to (-2,0);
			\draw [bend left=20,line width=0.8pt, black] (2,0) to (-2,0);
			\draw [bend right=20,line width=0.8pt, black] (2,0) to (-2,0);
               
            \draw [fill=white,line width=0.8pt] (0,3) node[above] {$v$} circle (3pt); 
            \draw [fill=black,line width=0.8pt] (2,0) node[right=0.5mm] {$y$} circle (4pt); 
            \draw [fill=black,line width=0.8pt] (-2,0) node[left=0.5mm] {$x$} circle (4pt); 
		\end{tikzpicture}
		\caption{$T_{3,3,5}$}
		\label{fig:T335}
	\end{subfigure} 
	\caption{\small The graphs $T^{o}_{2,2,6}$ in~\Cref{cla:T226o} and $T_{3,3,5}$ in~\Cref{cla:T335}}
	\label{fig:configurations_pairs}
\end{figure}

To prove our next result, we again use \Cref{lem:14_lift} to handle vertices in $G$ of small degree.

\begin{lemma}\label{cla:T335}
	$G$ contains no $T_{3,3,5}$.
\end{lemma}

\begin{proof}
	Suppose $G$ contains $T_{3,3,5}$ as a subgraph as shown in~\Cref{fig:T335}. 
    We lift three edge-pairs $(xv,vy)$ at $v$, contract the resulting $8K_2$, denote by $w_{xy}$ the new vertex 
    formed by contraction, and denote the resulting graph by $G'$. 
    Note that $v(G')=v(G)-1\geq 4$ and $e(G')=e(G)-11$. 
    For the trivial partition $\mathcal{P}'_0$ of $V(G')$, we have $w_{G'}(\mathcal{P}'_0)\geq w(G)-2\times 11+23=1$. 
    Hence, $G'/\mathcal{P}'_0\notin \mathcal{N}\cup\mathcal{W}^*$ since $|\mathcal{P}'_0|\geq 4$. 
    Let $\mathcal{P}'$ be a nontrivial partition of $V(G')$ and $\mathcal{P}$ be the corresponding partition of $V(G)$. 
    If $w_{xy}$ and $v$ are in the same part of $\mathcal{P}'$, then $\mathcal{P}$ has type $(3^+,1^+,\ast)$; thus 
    \Cref{lem:partition_type_value}\ref{lem:2_partition_type_value_3111} implies
    $w_{G'}(\mathcal{P}')=w_G(\mathcal{P})\geq 16$, so \Cref{ob:partition_bound_notin}\ref{ob:1_14_WN} gives
    $G'/\mathcal{P}'\notin\mathcal{N}\cup\mathcal{W}^*$.
    From now on, we assume that $w_{xy}$ and $v$ are in different parts of $\mathcal{P}'$. We consider the following two types of partitions.
 
    {\bf Case 1: $\bm{\mathcal{P}'}$ has type $\bm{(2,1,\ast)}$.} Now $|\mathcal{P}'|\geq 3$ since $v(G')\ge 4$; 
    so~\Cref{lem:partition_type_value}\ref{lem:2_partition_type_value_3111} and \ref{lem:3_partition_type_value_2211} 
    give $w_{G'}(\mathcal{P}')\geq w_G(\mathcal{P})-2\times 6\geq\min\{16,18\}-12=4$.
    Thus $G'/\mathcal{P}'\notin \mathcal{N}\cup\mathcal{W}^*\backslash\{T_{a,b,c}:a+b+c=16\}$. 
    Suppose $\mathcal{P}'=\{P_1, P_2, P_3\}$ with $|P_1|=2$ and $ |P_2|=|P_3|=1$, and $G'/\mathcal{P}'=\{T_{a,b,c}:a+b+c=16\}$. 
    Denote $V(G')$ by $\{w_{xy},v,v_1,v_2\}$. Note that $v(G)=5$ and $0\le w(G)\le 2e(G)-5\times 23+42$ gives 
    $e(G)\geq \lceil73/2\rceil = 37$. 
    If $\mathcal{P}'=\{\{w_{xy}\},\{v,v_1\},\{v_2\}\}$, then $\mu_{G}(v_1v)=e(G)-e(G'/\mathcal{P}')-e(T_{3,3,5})\geq 37-16-11 = 10$. 
    So $G$ has an $\mathcal{N}$-good subgraph $\alpha K_2$ with $\alpha\geq 10$, contradicting~\Cref{thm:main_statement_SZ9_k8}\ref{thm:1_main_statement_SZ9_k8}. 
    Similarly, in the case $\mathcal{P}'=\{\{w_{xy}\},\{v\},\{v_1,v_2\}\}$ we get $\mu_{G}(v_1v_2)\ge 10$, a contradiction. In the remaining case $\mathcal{P}'=\{\{w_{xy},v_1\},\{v\},\{v_2\}\}$. Now $\mu_{G}(v_1x)+\mu_{G}(v_1y)
    = e(G)-e(G'/\mathcal{P}')-e(T_{3,3,5})\geq 10$. 
    Noting that $\mu_G(vx)\geq 3$ and $\mu_G(vy)\geq 3$, since $G$ contains no $T^o_{2,2,6}$, the vertex $v_1$ cannot be adjacent to both $x$ and $y$; that is, either $\mu_{G}(v_1x)=0$ 
    or $\mu_{G}(v_1y)=0$.  Thus $G$ again contains an $\mathcal{N}$-good subgraph $\alpha K_2$ with $\alpha\geq 10$ of $G$,  
    again contradicting~\Cref{thm:main_statement_SZ9_k8}\ref{thm:1_main_statement_SZ9_k8}.

    {\bf Case 2: $\mathcal{P}'$ has type $\bm{({3^+},{1^+},*)}$ or type $\bm{({2^+},{2^+},*)}$.} 
    Now $\mathcal{P}$ has either type $(4^+,1^+,\ast)$ or type $(3^+,2^+,\ast)$. 
    So~\Cref{lem:partition_type_value}\ref{lem:4_partition_type_value_4111} and \ref{lem:5_partition_type_value_3211} 
    imply $w_{G'}(\mathcal{P}')\geq w_G(\mathcal{P})-2\times 6\geq \min \{20,25\}-12=8$.
    We now bound the edge-connectivity of $G'$. By~\Cref{lem:edge_connectivity_14}, $d_{G'}(v)=d_{G}(v)-6\geq 14-6=8$. 
    Since $G$ contains no $T_{2,2,6}$, we know $\mu_{G}(xy)=5$; hence, $d_{G'}(w_{xy})=d_{G}(x)+d_{G}(y)-2\times\mu_{G}(xy)-6\geq 
    14+14-2\times5-6=12$. And $d_{G'}(u)=d_{G}(u)\geq 14$ for every vertex $u\in V(G')\backslash \{v,w_{xy}\}$. 
    If $[X, X^c]$ is an edge-cut of $G'$ with $|X|\ge2$ and $|X^c|\ge2$, then~\Cref{lem:edge_connectivity}\ref{lem:edge_cut_14} 
    implies $d_{G'}(X)\ge d_G(X)-6\geq 15-6=9$. Thus $d_{G'}(X)\geq 9$ unless $\{X,X^c\}=\{\{v\},V(G')\backslash \{v\}\}$. 
    So~\Cref{ob:partition_bound_notin}\ref{ob:5_10-edge-connected} gives $G'/\mathcal{P}'\notin\mathcal{N}
    \cup\mathcal{W}^*\backslash \{8K_2\}$ and $G'/\mathcal{P}'=8K_2$ only if $\mathcal{P}'=\{\{v\},V(G')\setminus\{v\}\}$. 
    In this exceptional case, $d_{G}(v)=14$. Now we instead lift three edge-pairs $(xv, vy)$ to obtain a new graph $G_1$, 
    delete vertex $v$, and denote the resulting graph by $G_2$. Note that $\mu_{G_2}(u_1u_2)\leq 8$ for every two vertices $u_1,u_2\in V(G_2)$, i.e., $G_2$ contains no $9K_2$. By~\Cref{lem:14_lift}, $G_2$ is $\mathcal{S}$-good and $G_1/\mathcal{P}\notin\mathcal{N}$ for every partition $\mathcal{P}$ of $V(G_1)$. This contradicts~\Cref{thm:main_statement_SZ9_k8}\ref{thm:3_main_statement_SZ9_k8}.

	Therefore, $w(G')\geq 1$ and $G'/\mathcal{P}'\notin\mathcal{N}\cup\mathcal{W}^*$ for every partition $\mathcal{P}'$ of $V(G')$. So $G'$ is $\mathcal{S}$-good, contradicting~\Cref{thm:main_statement_SZ9_k8}\ref{thm:2_main_statement_SZ9_k8}. 
\end{proof}

\begin{figure}[!htbp]
	\centering
	\begin{subfigure}[t]{.32\textwidth}
		\centering
		\begin{tikzpicture}[scale=.56]			
			\draw [bend left=4, line width=0.2mm, black] (-2,-2) to (2,-2);%v1v2
			\draw [bend right=4, line width=0.2mm, black] (-2,-2) to (2,-2);
			\draw [bend left=12, line width=0.2mm, black] (-2,-2) to (2,-2);
			\draw [bend right=12, line width=0.2mm, black] (-2,-2) to (2,-2);
			\draw [bend left=20, line width=0.2mm, black] (-2,-2) to (2,-2);
			\draw [bend right=20, line width=0.2mm, black]  (-2,-2) to (2,-2);

			\draw [bend left=4, line width=0.2mm, black] (-2,-2) to (-2,-6);%v1v4
			\draw [bend right=4, line width=0.2mm, black] (-2,-2) to (-2,-6);
			\draw [bend left=12, line width=0.2mm, black] (-2,-2) to (-2,-6);
			\draw [bend right=12, line width=0.2mm, black] (-2,-2) to (-2,-6);  
			\draw [bend left=20, line width=0.2mm, black] (-2,-2) to (-2,-6);
			\draw [bend right=20, line width=0.2mm, black] (-2,-2) to (-2,-6);

			\draw [bend left=4, line width=0.2mm, black] (-2,-6) to (2,-6);%v4v3
			\draw [bend right=4, line width=0.2mm, black] (-2,-6)  to (2,-6);
			\draw [bend left=12, line width=0.2mm, black] (-2,-6)  to (2,-6);
			\draw [bend right=12, line width=0.2mm, black] (-2,-6)  to (2,-6);  
			\draw [bend left=20, line width=0.2mm, black] (-2,-6)  to (2,-6);
			\draw [bend right=20, line width=0.2mm, black] (-2,-6)  to (2,-6);

			\draw [bend left=4, line width=0.2mm, black] (2,-2) to (2,-6);%v2v3
			\draw [bend right=4, line width=0.2mm, black] (2,-2)  to (2,-6);
			\draw [bend left=12, line width=0.2mm, black] (2,-2)  to (2,-6);
			\draw [bend right=12, line width=0.2mm, black] (2,-2)  to (2,-6);
			\draw [bend left=20, line width=0.2mm, black] (2,-2)  to (2,-6);
			\draw [bend right=20, line width=0.2mm, black] (2,-2)  to (2,-6);
			
			\draw [line width=0.2mm, black] (-2,2) to (2,2);
			\draw [line width=0.2mm, black] (2,2) to (2,-2);
			\draw [line width=0.2mm, black] (-2,2) to (-2,-2);%xv1

   		    \draw [fill=black,line width=0.2pt] (-2,-2) node[left=0.5mm] {$x$} circle (4pt); 
            \draw [fill=black,line width=0.2pt] (2,-2) node[right=0.5mm] {$y$} circle (4pt); 
            \draw [fill=white,line width=0.8pt] (2,-6) node[right=0.5mm] {$v$} circle (3.5pt); 
            \draw [fill=white,line width=0.8pt] (-2,-6) node[left=0.5mm] {$u$} circle (3.5pt); 
            \draw [fill=white,line width=0.8pt] (-2,2) node[left=0.5mm] {$v_1$} circle (3.5pt); 
            \draw [fill=white,line width=0.8pt] (2,2) node[right=0.5mm] {$v_2$} circle (3.5pt); 
		\end{tikzpicture}
		\caption{$Q^{o}_{6,6,6,7}$}
		\label{fig:Q6667o}	
	\end{subfigure}
	\begin{subfigure}[t]{.32\textwidth}
		\centering
		\begin{tikzpicture}[scale=.56]		
			\draw [line width=0.2mm, black] (-2,1)  to (2,1);%xy			
			\draw [line width=0.2mm, black] (-2,1)  to (2,1);
			\draw [bend left=10, line width=0.2mm, black] (-2,1) to (2,1);
			\draw [bend right=10, line width=0.2mm, black] (-2,1) to (2,1);
			\draw [bend left=18, line width=0.2mm, black] (-2,1) to (2,1);
			\draw [bend right=18, line width=0.2mm, black] (-2,1) to (2,1);			
			
			\draw [bend left=12, line width=0.2mm, black] (-2,1) to (0,5);%xz
			\draw [bend right=12, line width=0.2mm, black] (-2,1) to (0,5);
			
			\draw [bend left=4, line width=0.2mm, black] (-2,1) to (-2,-3);%xu
			\draw [bend right=4, line width=0.2mm, black] (-2,1) to (-2,-3);
			\draw [bend left=12, line width=0.2mm, black] (-2,1) to (-2,-3);
			\draw [bend right=12, line width=0.2mm, black] (-2,1) to (-2,-3);
			\draw [bend left=20, line width=0.2mm, black] (-2,1) to (-2,-3);
			\draw [bend right=20, line width=0.2mm, black] (-2,1) to (-2,-3);
			
			\draw [bend left=12, line width=0.2mm, black] (2,1) to (0,5);%yz
			\draw [bend right=12, line width=0.2mm, black] (2,1) to (0,5);
			
			\draw [bend left=4, line width=0.2mm, black] (2,-3) to (-2,-3);%uv
			\draw [bend right=4, line width=0.2mm, black] (2,-3) to (-2,-3);
			\draw [bend left=12, line width=0.2mm, black] (2,-3) to (-2,-3);
			\draw [bend right=12, line width=0.2mm, black] (2,-3) to (-2,-3);
			\draw [bend left=20, line width=0.2mm, black] (2,-3) to (-2,-3);
			\draw [bend right=20, line width=0.2mm, black] (2,-3) to (-2,-3);
			
			\draw [bend left=4, line width=0.2mm, black] (2,-3) to (2,1);%uy
			\draw [bend right=4, line width=0.2mm, black] (2,-3) to (2,1);
			\draw [bend left=12, line width=0.2mm, black] (2,-3) to (2,1);
			\draw [bend right=12, line width=0.2mm, black] (2,-3) to (2,1);
			\draw [bend left=20, line width=0.2mm, black] (2,-3) to (2,1);
			\draw [bend right=20, line width=0.2mm, black] (2,-3) to (2,1);
            
            \draw [fill=white,line width=0.8pt] (0,5) node[above] {$z$} circle (3.5pt); 
            \draw [fill=black,line width=0.8pt] (-2,1) node[left=0.5mm] {$x$} circle (4pt); 
            \draw [fill=black,line width=0.8pt] (2,1) node[right=0.5mm] {$y$} circle (4pt); 
            \draw [fill=white,line width=0.8pt] (-2,-3) node[left=0.5mm] {$u$} circle (3.5pt); 
            \draw [fill=white,line width=0.8pt] (2,-3) node[right=0.5mm] {$v$} circle (3.5pt); 
		\end{tikzpicture}
		\caption{$Q_{6,6,6,7}^{oo}$}
		\label{fig:Q6667_oo}		
	\end{subfigure}
	\begin{subfigure}[t]{.32\textwidth}
		\centering
		\begin{tikzpicture}[circle dotted/.style={dash pattern=on 0.2mm off 1mm, line cap=round}, scale=.56]		
			
			\draw [bend left=8, line width=0.2mm, black] (0,3) to (2,0);%xy
			\draw [bend right=8, line width=0.2mm, black] (0,3) to (2,0);
			\draw [bend left=20, line width=0.2mm, black] (0,3) to (2,0);
			\draw [bend right=20, line width=0.2mm, black] (0,3) to (2,0);
			
			\draw [bend left=8, line width=0.2mm, black] (0,3) to (-2,0);%xz
			\draw [bend right=8, line width=0.2mm, black] (0,3) to (-2,0);
			\draw [bend left=20, line width=0.2mm, black] (0,3) to (-2,0);
			\draw [bend right=20, line width=0.2mm, black] (0,3) to (-2,0);
			
			\draw [line width=0.2mm, black] (0,3) to (-4,3.5) to (-2,0);	%xuz
			
			\draw [bend left=8, line width=0.2mm, black] (0,3) to (4,3.5) ; %xv2
			\draw [bend right=8, line width=0.2mm, black] (0,3) to (4,3.5) ;
			
			\draw [bend left=8, line width=0.2mm, black] (2,0) to (-2,0); %yz
			\draw [bend right=8, line width=0.2mm, black] (2,0) to (-2,0);
			\draw [bend left=20, line width=0.2mm, black] (2,0) to (-2,0);
			\draw [bend right=20, line width=0.2mm, black] (2,0) to (-2,0);
			
			\draw [bend left=10, line width=0.2mm, black] (2,0) to (4,3.5) ; %yv2
			\draw [bend right=10, line width=0.2mm, black] (2,0) to (4,3.5) ;
			
			\draw [line width=0.2mm, black] (-2,0) to (0,-3) to (2,0); %zv3y

            \draw [fill=white,line width=0.8pt] (-4,3.5) node[above] {$v_1$} circle (3.5pt);    
            \draw [fill=white,line width=0.8pt] (4,3.5) node[above] {$v_2$} circle (3.5pt);    
            \draw [fill=white,line width=0.8pt] (0,-3) node[right=0.5mm] {$v_3$} circle (3.5pt); 
            \draw [fill=black,line width=0.4pt] (0,3) node[above] {$x$} circle (4pt);  
            \draw [fill=black,line width=0.2pt] (2,0) node[right=0.5mm] {$y$} circle (4pt); 
            \draw [fill=black,line width=0.2pt] (-2,0) node[left=0.5mm] {$z$} circle (4pt); 
		\end{tikzpicture}
		\caption{$F$}
		\label{fig:F}		
	\end{subfigure}
	\caption{\small The graphs $Q_{6,6,6,7}^o$ and $Q_{6,6,6,7}^{oo}$ in~\Cref{cla:Q6667o}, and $F$ in~\Cref{cla:F1}}
	\label{fig:Q6667o_Q6667oo_F}
\end{figure}
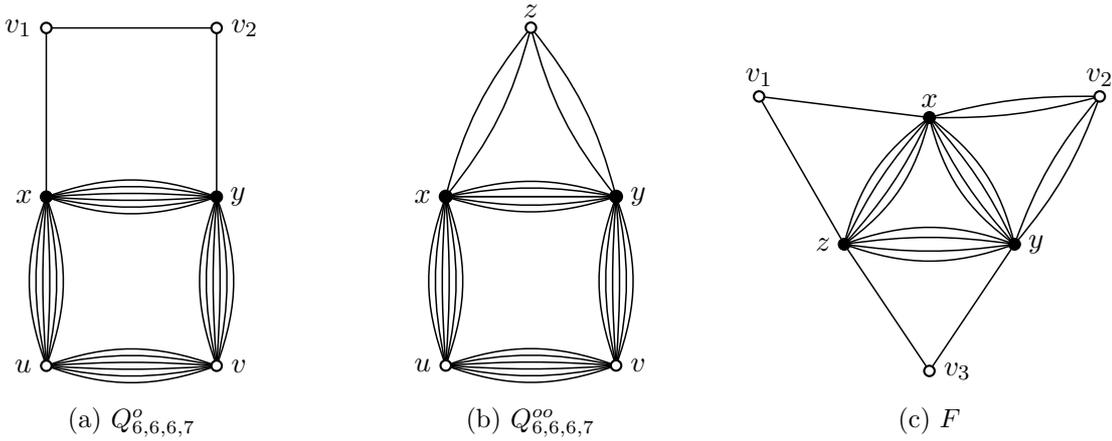

\begin{lemma}\label{cla:Q6667o}
	$G$ contains neither $Q^{o}_{6,6,6,7}$ nor $Q^{oo}_{6,6,6,7}$.
\end{lemma}

\begin{proof}
	Suppose to the contrary that $G$ contains $Q^{o}_{6,6,6,7}$ or $Q^{oo}_{6,6,6,7}$ as shown in~\Cref{fig:Q6667o} or~\ref{fig:Q6667_oo}. The proofs for the two configurations are quite similar. We lift some edges in $G$ to form the subgraph $6C_4^+$, 
    which is $\mathcal{N}$-good, contract this $6C_4^+$, and verify that the resulting graph is $\mathcal{S}$-good. 
    This contradicts~\Cref{thm:main_statement_SZ9_k8}\ref{thm:2_main_statement_SZ9_k8}. 
    Here we only give the proof for $Q^{oo}_{6,6,6,7}$.
 
	We lift two edge-pairs $(xz,zy)$ at vertex $z$, then contract the resulting $6C_4^+$ into a new vertex $w_{xyvu}$, and denote the resulting graph by $G'$. 
    Note that $v(G')=v(G)-3\geq 2$. 
    And $\mu(xv)=\mu(uy)=0$, since otherwise $G$ contains $T^o_{2,2,6}$, contradicting 
    \Cref{cla:T226o}. Hence, $e(G')=e(G)-27$. Furthermore, as we only lifted 4 edges, 
    \Cref{lem:edge_connectivity_14} implies that $G'$ has edge-connectivity at least $14-4=10$.
    
    For the trivial partition $\mathcal{P}'_0$ of $V(G')$, we have $w_{G'}(\mathcal{P}'_0)\geq w(G)-2\times 27+23\times 3\geq 15$ so \Cref{ob:partition_bound_notin}\ref{ob:1_14_WN} implies $G'/\mathcal{P}'_0\notin\mathcal{N}\cup\mathcal{W}^*$. 
    Let $\mathcal{P}'$ be a type $(2^+,1^+,\ast)$ partition of $V(G')$ and $\mathcal{P}$ be the corresponding partition of $V(G)$. 
    Note that $\mathcal{P}$ has type either $(5^+,1^+,\ast)$ or $(4^+,2^+,\ast)$; 
    by~\Cref{lem:partition_type_value}\ref{lem:4_partition_type_value_4111} we get $w_{G}(\mathcal{P})\geq 20$. 
    Thus, $w_{G'}(\mathcal{P}')\geq w_G(\mathcal{P})-2\times 4\geq 12$. 
    Since $G'$ is $10$-edge-connected, \Cref{ob:partition_bound_notin}\ref{ob:5_10-edge-connected} gives 
    $G'/\mathcal{P}'\notin\mathcal{N}\cup\mathcal{W}^*$. In summary, $w(G')\geq 12$; so $G'$ is $\mathcal{S}$-good, contradicting~\Cref{thm:main_statement_SZ9_k8}\ref{thm:2_main_statement_SZ9_k8}.
\end{proof}

\begin{lemma}\label{cla:F1}
	$G$ contains no copy of $F$.
\end{lemma}

\begin{proof}
	Suppose to the contrary that $G$ contains a subgraph $F$ on six vertices, as shown in~\Cref{fig:F}. Note that in $F$, all three $v_i$'s are distinct and thus $v(G)\geq 6$.
    We lift an edge-pair $(xv_1,v_1z)$, two edge-pairs $(xv_2,v_2y)$, and an edge-pair $(yv_3,v_3z)$,  to form the subgraph $T_{5,5,6}$ which is $\mathcal{N}$-good (by \Cref{ob:W*_N_good}). Next, we contract $T_{5,5,6}$ into a new vertex $w_{xyz}$ and denote the resulting graph by $G'$. Note that $v(G')\geq 4$ and $e(G')=e(G)-20$. 
    \Cref{lem:edge_connectivity_14} gives $d_{G'}(v)\geq d_{G}(v)-4\geq 10$ for every $v\in V(G')\setminus\{w_{xyz}\}$ and~\Cref{lem:edge_connectivity}\ref{lem:edge_cut_18}  implies that $d_{G'}(w_{xyz})\ge 18-8=10$.
    Now consider an edge cut $[X',X'^c]$ of $G'$ with $|X'|\ge 2$ and $|X'^c|\ge 2$.  By symmetry, we assume $w_{xyz}\in X'$.  If $v_1,v_2,v_3\in X'^c$, then for the corresponding edge-cut $[X,X^c]$ in $G$ we have $|X|\ge 3$ and $|X^c|\ge 3$, so 
    \Cref{lem:edge_connectivity}\ref{lem:edge_cut_18} gives $e_{G'}(X',X'^c) \ge e_G(X,X^c)-8\ge 18-8=10$.
    Otherwise, at most 6 of the edges that we lifted are counted in $e_G(X,X^c)$ but not in $e_{G'}(X',X'^c)$, so \Cref{lem:edge_connectivity}\ref{lem:edge_cut_14} gives $e_{G'}(X',X'^c)\ge e_G(X,X^c)-6\ge 15-6=9$.
    Thus $G'$ is $9$-edge-connected.  
    
    For the trivial partition $\mathcal{P}_0'$ of $V(G')$, we have $w_{G'}(\mathcal{P}_0')\geq w(G)-2\times 20+23\times 2\geq6$. 
    Since $|\mathcal{P}'_0|=v(G')\geq 4$, clearly $G'/\mathcal{P}'_0\notin\mathcal{N}\cup\mathcal{W}^*$. 
    Given a type $(2^+,1^+,\ast)$ partition $\mathcal{P}'$ of $V(G')$, denote the corresponding partition of $V(G)$ 
    by $\mathcal{P}$. If at least one $v_i$ is in the same part of $\mathcal{P}'$ as $w_{xyz}$, then $\mathcal{P}$ has type $(4^+,1^+,\ast)$, so \Cref{lem:partition_type_value}\ref{lem:4_partition_type_value_4111} gives $w_{G'}(\mathcal{P}')\geq w_G(\mathcal{P})-2\times \max\{4,6\}\geq 20-12=8$. 
    If two $v_i$'s are in the same part of $\mathcal{P}'$, different from the part of $w_{xyz}$, then $\mathcal{P}$ has type $(3^+,2^+,\ast)$, so \Cref{lem:partition_type_value}\ref{lem:5_partition_type_value_3211} gives $w_{G'}(\mathcal{P}')\ge w_G(\mathcal{P})-2\times 8\ge 25-16=9$. In both cases,
    as $G'$ is $9$-edge-connected, \Cref{ob:partition_bound_notin}\ref{ob:5_10-edge-connected} gives
    $G'/\mathcal{P}'\notin\mathcal{N}\cup\mathcal{W}^*$. Assume instead that $v_1$, $v_2$, $v_3$, and $w_{xyz}$ are in four different parts of $\mathcal{P}'$. 
    Only the 8 lifted edges between $\{v_1,v_2,v_3\}$ and $\{x,y,z\}$ are counted in $w_{G}(\mathcal{P})$ but not in $w_{G'}(\mathcal{P}')$. 
    Hence, $w_{G'}(\mathcal{P}')\geq w_G(\mathcal{P})-2\times 8\geq20-16=4$. But in this case, 
    $|\mathcal{P}'|\geq 4$; thus $G'/\mathcal{P}'\notin\mathcal{N}\cup\mathcal{W}^*$. 
    
    Altogether, $w(G')\geq 4$. Hence, $G'$ is $\mathcal{S}$-good, contradicting 
    \Cref{thm:main_statement_SZ9_k8}\ref{thm:2_main_statement_SZ9_k8}.
\end{proof}

The proof of our final lemma relies on the fact that $G$ is planar.

\begin{lemma}\label{lem:T_444}
Assume that $G$ contains a face $f$ that is the inner face of $T_{4,4,4}$. If $f$ is weakly adjacent to three $3$-faces $f_1, f_2$, and $f_3$, then for all distinct $i,j\in[3]$, faces $f_i$ and $f_j$ cannot be weakly adjacent.
\end{lemma}

\begin{proof}
     Assume to the contrary, without loss of generality,  that $f_1$ and $f_2$ are weakly adjacent. Let $V(T_{4,4,4})=\{x,y,z\}$ and let $w$ be the vertex such that $f_1$ and $f_2$ are weakly adjacent across the multiedge $xw$. Thus there is a multi-$K_4$ induced on $x,y,z,w$. We fix the planar embedding such that $x$ is the interior vertex and $f_1, f_2$, and $f$ are three facial $3$-cycles containing $x$. Since $G$ contains no $T_{3,3,5}$ (by \Cref{cla:T335}), $\mu_G(xy)=\mu_G(yz)=\mu_G(xz)=4$. Since $G$ contains no $T^o_{1,1,7}$ (by \Cref{cla:configurations_pairs_at most5}), we have $\mu_{G}(xw) \leq 5$. We note that $x$ is the interior vertex, so by planarity $N_G(x)=\{w,y,z\}$. By~\Cref{lem:edge_connectivity_14}, $d_{G}(x)\geq \delta(G)\geq 14$. However, $d_{G}(x)=\mu_G(xw)+\mu_G(xy)+\mu_G(xz)\leq 5+4+4=13$, a contradiction. 
\end{proof}

\subsection{Discharging part}
\label{sec:discharging}
By all the previous lemmas, $G$ contains no configuration in the set $$\{8K_2, T_{1,1,7}, T_{2,2,6}, T_{3,3,5},  T_{1,1,7}^o, T_{2,2,6}^o, Q_{1,1,1,7}, Q_{1,1,1,7}^o,  Q_{6,6,6,7}^o, Q_{6,6,6,7}^{oo}, V_{1,1,1,1,7}, F\}.$$

Note that our minimal counterexample $G$ satisfies $\omega(G)\geq 0$, so $2e(G)\geq 23v(G)-42$. 
Using this to substitute for $v(G)$ in Euler's formula that $v(G)+f(G)-e(G)=2$, we get

\begin{equation}\label{equ:discharge}
    \sum\limits_{f\in F(G)}d(f)=2e(G)\leq (2+\frac{4}{21})f(G)-\frac{8}{21}.   
\end{equation}

We assign to each face $f$ an initial charge $d(f)$; thus the total initial charge is strictly smaller than $\frac{46}{21}f(G)$. We then apply the following discharging rules to redistribute the charges among the faces.
	
\medskip
\noindent
{\bf Rule~(A)}. \emph{Each $3^+$-face gives charge $\frac{2}{21}$ to each of its weakly adjacent $2$-faces.}
	
\medskip
\noindent{\bf Rule~(B)}. \emph{Every inner $4$-face of $Q_{a,b,c,d}$ with $a+b+c+d\leq 22$ gives charge $\frac{1}{21}$ to each of its weakly adjacent $3$-faces.}
\emph{Every inner $3$-face of $T_{a,b,c}$ with $a+b+c\le 11$ gives charge $\frac{1}{42}$ to each of its weakly adjacent $3$-faces across each edge with multiplicity at least $4$.}	

\medskip
\noindent
{\bf Rule~(C)}. \emph{Each $5^+$-face gives charge $\frac{9}{105}$ to each of its weakly adjacent $3$-faces and $4$-faces.}
\medskip
	
We now prove that after discharging each face ends with charge at least $\frac{46}{21}$, which is a contradiction. 

\smallskip
By {\bf Rule~(A)}, every 2-face $f$ receives $\frac{2}{21}$ from each of its two weakly adjacent $3^+$-faces; thus $f$ ends with 
at least $2+2\times\frac{2}{21}=2+\frac{4}{21}$.
	
We first consider a $6^+$-face $f$. Since $G$ contains no $8K_2$, $f$ has at most $6d(f)$ weakly adjacent $2$-faces. Moreover, since $G$ contains no $T_{1,1,7}$ and no $Q_{1,1,1,7}$, $f$ sends charge in total at most $d(f)\times\max\{6\times\frac{2}{21}, 
5\times\frac{2}{21}+\frac{9}{105}\}=\frac{12}{21}d(f)$ to its weakly adjacent $2$-faces, $3$-faces, and $4$-faces 
by {\bf Rules~(A) and (C)}.  Thus $f$ ends with at least $d(f)-\frac{12}{21}d(f)=\frac{9}{21}d(f)>\frac{46}{21}$.
	
Next, we consider an inner $5$-face $f$ of the subgraph $V_{a,b,c,d,e}$ (a multi-$C_5$). Note that $V_{a,b,c,d,e}$ with $a+b+c+d+e\geq 31$ must contain $7K_2$ as a subgraph, and hence contain $V_{1,1,1,1,7}$, contradicting~\Cref{cla:configurations_pairs_at most5}. 
Thus $a+b+c+d+e\leq 30$. By {\bf Rules~(A) and (C)}, the face $f$ ends with at least $5-25\times\frac{2}{21}-5\times\frac{9}{105}=\frac{46}{21}$. 
	
Next, we consider an inner $4$-face $f$ of the subgraph $Q_{a,b,c,d}$. Since $G$ contains no $Q_{1,1,1,7}$, $a+b+c+d\leq 4\times 6=24$ and $\mu\leq 6$. Moreover, since $G$ contains no $Q^o_{1,1,1,7}$, $f$ cannot be weakly adjacent to any $3$-face across an edge with multiplicity $6$. We consider the cases based on the value of $a+b+c+d$. Recall that {\bf Rule (B)} applies only if $a+b+c+d\leq 22$.
\begin{itemize}[itemsep=-2pt]
     \item If $a+b+c+d\leq 21$, then by {\bf Rules~(A) and (B)}, $f$ ends with a charge of at least $4-(a+b+c+d-4)\times\frac{2}{21}-4\times \frac{1}{21}=\frac{88-2(a+b+c+d)}{21}\geq \frac{46}{21}$.

     \item If $a+b+c+d= 22$, then we only need to consider $Q_{5,5,6,6}$, $Q_{5,6,5,6}$, and $Q_{4,6,6,6}$. Note that if $f$ is of $Q_{5,5,6,6}$ or $Q_{5,6,5,6}$, then $f$ is weakly adjacent to at most two $3$-faces; and if $f$ is of $Q_{4,6,6,6}$, then $f$ is weakly adjacent to at most one $3$-face. By {\bf Rules~(A) and (B)}, then $f$ ends with at least $4-18\times\frac{2}{21}-\max\{1,2\}\times \frac{1}{21}=\frac{46}{21}$. 

    \item If $a+b+c+d=23$, then the $4$-face $f$ must be the inner face of $Q_{5,6,6,6}$, and by {\bf Rule~(A)}, $f$ ends with 
at least $4-19\times \frac{2}{21}=\frac{46}{21}$. 
     
    \item If $a+b+c+d=24$, then we just need to consider $Q_{6,6,6,6}$. Because $G$ contains no $Q_{1,1,1,7}^o$ and no $Q_{6,6,6,7}^o$, the $4$-face $f$ of $Q_{6,6,6,6}$ must be weakly adjacent to four $5^+$-faces. By {\bf Rules~(A) and (C)}, the face $f$ (of $Q_{6,6,6,6}$) ends with at least $4-20\times\frac{2}{21}+4\times\frac{9}{105}=\frac{256}{105}>\frac{46}{21}$.
\end{itemize}

Finally, we consider an inner $3$-face $f$ of the subgraph $T_{a,b,c}$. Since $G$ contains no $T_{1,1,7}$, $T_{2,2,6}$, and $T_{3,3,5}$, $a+b+c\leq \max\{1+6+6, 2+5+5, 4+4+4\}=13$ and $\mu\leq 6$. Moreover, since $G$ contains neither $T^o_{1,1,7}$ nor $Q^o_{1,1,1,7}$, $f$ cannot be weakly adjacent to any $3$-face or $4$-face across an edge with multiplicity $6$. We consider the cases based on the value of $a+b+c$.

\begin{itemize}[itemsep=-2pt]
    \item If $a+b+c\leq 11$, then at most two of $a$, $b$, and $c$ are larger than or equal to $4$. Recall that in {\bf Rule (B)} that every inner $3$-face of $T_{a,b,c}$ with $a+b+c\le 11$ gives charge $\frac{1}{42}$ to each of its weakly adjacent $3$-faces only across each edge with multiplicity at least $4$. By {\bf Rules~(A) and (B)}, the face $f$ ends with at least $3-(a+b+c-3)\times\frac{2}{21}-2\times\frac{1}{42}=\frac{68-2(a+b+c)}{21}\geq \frac{46}{21}$. 

    \item If $a+b+c\geq12$, then we only need to consider the following configurations: $T_{1,5,6}$, $T_{2,5,5}$, $T_{4,4,4}$, and $T_{1,6,6}$. 

    For $T_{1,5,6}$ and $T_{1,6,6}$, the face $f$ of $T_{1,5,6}$ or $T_{1,6,6}$ is weakly adjacent with $5^+$-faces across each edge with multiplicity $6$. So by {\bf Rules~(A) and (C)}, each $f$ of $T_{1,5,6}$ ends with at least $3-9\times\frac{2}{21}+\frac{9}{105}=\frac{234}{105}>\frac{46}{21}$ and each $f$ of $T_{1,6,6}$ ends with at least $3-10\times\frac{2}{21}+2\times\frac{9}{105}=\frac{233}{105}>\frac{46}{21}$. 
    
    For $T_{2,5,5}$, since $G$ has neither $T_{2,2,6}^o$ nor $Q_{6,6,6,7}^{oo}$, each $f$ of $T_{2,5,5}$ is weakly adjacent (across each edge of multiplicity 5) with a $5^+$-face or an inner $4$-face of the subgraph $Q_{a,b,c,d}$ with $a+b+c+d\leq 5+5+6+6=22$. Thus by {\bf Rules~(A), (B), and (C)}, each $f$ of $T_{2,5,5}$ ends with at least $3-9\times\frac{2}{21}+2\times\min\{\frac{1}{21},\frac{9}{105}\}>\frac{46}{21}$. 
    
    Finally, consider an inner $3$-face $f$ of the subgraph $T_{4,4,4}$.  By~\Cref{lem:T_444}, if $f$ of $T_{4,4,4}$ is weakly adjacent to three $3$-faces, then they are pairwise not weakly adjacent to each other. Since $G$ contains no copy of $F$, either (a) some face $f'$ that is weakly adjacent to $f$ is a $5^+$-face or a $4$-face of $Q_{a,b,c,d}$ with $a+b+c+d\leq 4+6+6+6=22$; or (b) every face $f'$ that is weakly adjacent to $f$ is the inner $3$-face of $T_{a,b,c}$ with $a+b+c\le 1+4+6=11$. Hence, by {\bf Rules~(A), (B), and (C)}, $f$ finishes with at least $3-9\times\frac{2}{21}+\min\{\frac{1}{21},\frac{9}{105},3\times\frac{1}{42}\}=\frac{46}{21}$.
\end{itemize}
	
We are done.

\section*{Acknowledgments}

Jiaao Li is supported by National Key Research and Development Program of China (No. 2022YFA1006400), National Natural Science Foundation of China (Nos. 12222108, 12131013), Natural Science Foundation of Tianjin (No. 22JCYBJC01520), and the Fundamental Research Funds for the Central Universities, Nankai University.
Zhouningxin Wang is supported by National Natural Science Foundation of China (Nos. 12301444) and the Fundamental Research Funds for the Central Universities, Nankai University.

\end{document}